\documentclass[11pt]{article}
\usepackage[margin=1in]{geometry}
\usepackage{url}
\usepackage{breakurl}
\usepackage{amssymb}
\usepackage{amsmath}
\usepackage{amsthm}
\usepackage{graphicx}
\usepackage[export]{adjustbox}
\usepackage{placeins}
\usepackage[english]{babel}
\newtheorem{theorem}{Theorem}[section]

\newtheorem{lemma}[theorem]{Lemma}

\usepackage[colorlinks=true,urlcolor=blue,citecolor=blue,linkcolor=blue]{hyperref}
\usepackage{xurl}
\usepackage{caption}
\usepackage{booktabs}
\usepackage{capt-of}
\usepackage{xcolor}
\usepackage{microtype}
\usepackage{listings}
\usepackage{algorithm}
\usepackage{algpseudocode}

\definecolor{codegreen}{rgb}{0.0,0.5,0.0}
\definecolor{codegray}{rgb}{0.4,0.4,0.4}
\definecolor{codepurple}{rgb}{0.47,0.0,0.53}
\definecolor{backcolour}{rgb}{0.97,0.97,0.97}

\lstdefinestyle{pythonstyle}{
    language=Python,
    backgroundcolor=\color{backcolour},   
    commentstyle=\color{codegreen}\itshape,
    keywordstyle=\color{blue}\bfseries,
    numberstyle=\tiny\color{codegray},
    stringstyle=\color{codepurple},
    basicstyle=\ttfamily\small,
    breakatwhitespace=false,         
    breaklines=true,                 
    captionpos=b,                    
    keepspaces=true,                 
    numbers=left,                    
    numbersep=8pt,                  
    showspaces=false,                
    showstringspaces=false,
    showtabs=false,                  
    tabsize=4,
    frame=tb,                         % Clean top and bottom rules
    framerule=0.8pt,
    rulecolor=\color{black}
}

\title{A Computational and Conformal Analysis of Multibrot Sets}
\author{
    Rohan Senapati%
    \thanks{\copyright~2026 Rohan Senapati. Licensed under CC BY-NC-ND 4.0}
}
\date{September 2026}

\begin{document}

\maketitle

\begin{abstract}
For the unicritical family \(f_{c,j}(z)=z^j+c\), the Multibrot sets \(M_j\) converge geometrically to the closed unit disk as \(j\to\infty\), but this limit does not quantify their finite-degree area behavior. We estimate \(\operatorname{Area}(M_j)\) using escape-time pixel counting while separating the effects of the iteration threshold \(K\) and spatial sampling density \(\rho\). The finite-iteration sets form a nested decreasing sequence, while grid refinement is treated empirically. Across \(K\in\{25,100,500,1000\}\) and \(\rho\in\{128,256,512,1024\}\), the final grid refinement changes the reported areas by at most \(1.1\times10^{-3}\). At \(K=1000\) and \(\rho=1024\), the estimates are well described by
\[
C(j)\approx \pi-2.49216j^{-0.57103}.
\]
Independently, Laurent coefficients of the exterior conformal map and Gronwall's area formula yield finite-truncation upper bounds
\[
U_{1000}(j)\approx \pi-2.07405j^{-0.643427}.
\]
Across all tested degrees, the grid estimates remain below the corresponding conformal bounds, and both sequences exhibit similar power-law decay of the area deficit from \(\pi\).
\end{abstract}

\section{Introduction}
For an integer $j\ge 2$, consider the unicritical family
\[
f_{c,j}(z)=z^j+c,
\qquad c\in\mathbb C,
\]
and its connectedness locus
\[
M_j=\{c\in\mathbb C:\text{the orbit of }0\text{ under }f_{c,j}\text{ is bounded}\}.
\]
The case $j=2$ is the classical Mandelbrot set. For larger $j$, the sets
retain the basic structure of unicritical parameter spaces while becoming
increasingly disk-like. Boyd and Schulz proved that $M_j$ converges to the
closed unit disk in the Hausdorff metric as $j\to\infty$
\cite{boydschulz2011}. This gives a precise geometric limiting picture, but it
does not provide a finite-degree quantitative law for the area estimates
obtained in computation.

The goal of this paper is to study that finite-degree area evolution with a
pipeline in which each stage addresses a different source of uncertainty.
First, we compute escape-time pixel-counting estimates of $\operatorname{Area}(M_j)$.
Second, we separate the two numerical approximation parameters: the iteration
threshold $K$ and grid density $\rho$. The $K$-dependence admits a clean
measure-theoretic interpretation because the exact finite-iteration sets are
nested, whereas the $\rho$-dependence is a discretization problem and is
therefore treated empirically. Third, we compare the resulting estimates with
rigorous conformal upper bounds obtained from Laurent coefficients and
Gronwall's area formula.

This separation is important. Increasing $K$ and increasing $\rho$ do not
represent the same limiting process, and numerical stabilization under either
parameter is not, by itself, a proof that pixel counting has converged to the
exact Lebesgue area. In particular, the irregular boundary prevents a naive
identification of grid refinement with a rigorous error bound. The numerical
study below therefore reports sensitivity and stabilization rather than a
claimed discretization theorem.

Using the finest tested computation, $K=1000$ and $\rho=1024$, we obtain a
degree-dependent empirical fit
\[
C(j)=\pi-\frac{2.49216}{j^{0.57103}},
\]
with $R^2=0.996$ over the fitted degrees. The purpose of this model is
descriptive: it quantifies the observed decay of the area deficit
$\pi-C(j)$ over the computed range. We then place these estimates beside
finite-truncation Gronwall bounds derived independently from the exterior
conformal map.

The main contributions are therefore:
\begin{enumerate}
    \item \textbf{A degree-dependent numerical area study.} Escape-time
    pixel counting is used to estimate $\operatorname{Area}(M_j)$ across
    polynomial degree and to quantify the observed approach toward $\pi$.

    \item \textbf{A two-parameter refinement analysis.} The iteration limit
    $K$ and spatial density $\rho$ are varied independently. The exact
    finite-iteration convergence is justified measure-theoretically, while
    the grid study is reported as empirical discretization stability.

    \item \textbf{An empirical area-deficit law.} The finest-grid estimates
    are summarized by a power-law model for $\pi-C(j)$, providing a compact
    quantitative description of the finite-degree trend.

    \item \textbf{An independent conformal comparison.} A Laurent-coefficient
    computation for arbitrary Multibrot degree is checked against known
    coefficient structure and combined with Gronwall's area formula to obtain
    rigorous upper bounds for comparison with the grid estimates.
\end{enumerate}

Reflectional and rotational symmetries are classical structural features of
Multibrot sets. Because they are useful pedagogically but are not the central
quantitative contribution of this paper, explicit derivations and associated
geometric visualizations are collected in Appendix~\ref{app:symmetry}.

The remainder of the paper follows the computational pipeline directly.
Section~\ref{sec:related} summarizes the relevant literature.
Section~\ref{sec:preliminaries} fixes notation and the escape criterion.
Section~\ref{sec:area} presents the area estimates and two-parameter refinement
study. Section~\ref{sec:scaling} gives the empirical degree-area model.
Section~\ref{sec:laurent} develops the conformal upper bounds and compares them
with the grid estimates.

\section{Related Work}\label{sec:related}

The Mandelbrot set and its higher-degree analogues arise naturally in the
study of unicritical polynomial dynamics. Foundational work of Douady and
Hubbard established the modern framework for connectedness loci and
polynomial-like mappings
\cite{douady1985,douadyhubbard1984,douadyhubbard1985};
subsequent work developed the structure of these parameter spaces in much
greater depth \cite{milnor2000,schleicher1999,cheraghi2007}.

For the large-degree regime, Boyd and Schulz proved Hausdorff convergence of
Multibrot sets to the closed unit disk and of the associated Julia sets to the
unit circle \cite{boydschulz2011}. The present work does not reprove that
geometric limit. Instead, it asks how finite-degree numerical area estimates
behave and how that behavior can be assessed through independent numerical
and conformal calculations.

Area is already subtle in the quadratic case. Ewing and Schober used
coefficients of the exterior uniformizing map together with Gronwall's area
theorem to obtain rigorous computational bounds for the Mandelbrot set
\cite{ewingschober1992}; later work refined numerical and analytic approaches
to the same problem \cite{bittner2017,siudem2023}. Direct lattice methods are
sensitive to the boundary, and Andreadis and Karakasidis later studied this
issue for the classical Mandelbrot set using a boundary-scanning approach that
distinguishes interior, exterior, and boundary lattice points
\cite{andreadis2015numerical}.

More directly related to the present study, Andreadis and Karakasidis
developed a finite-escape, lattice-based framework for numerically
approximating the areas of generalized Mandelbrot sets and investigated their
behavior under increasing polynomial degree, lattice resolution, and
iteration count \cite{andreadis2013generalized}. The present work likewise
uses direct pixel counting, but distinguishes the mathematical status of the
two approximation parameters: finite-iteration convergence is treated through
the nested sets $M_{j,K}$, whereas spatial grid refinement is interpreted only
as empirical discretization stability. Its principal additional component is
an independent conformal analysis: generalized Laurent coefficients are
combined with Gronwall's area formula to obtain upper bounds across degree,
which are then compared directly with the grid-based area estimates and their
observed degree-dependent scaling.

The conformal component builds on the exterior-map framework of Jungreis
\cite{jungreis1985} and on Shimauchi's generalization of Laurent-coefficient
structure to arbitrary Multibrot sets \cite{shimauchi2012exterior}. We use
these coefficients not as a primary object of arithmetic study, but as inputs
to Gronwall upper bounds that can be compared with the escape-time area
estimates.

Although the reflectional and rotational symmetries of Multibrot sets are
classical, Appendix~\ref{app:symmetry} develops a self-contained,
lemma-based framework for proving these invariances in representative cases.
The presentation separates modulus preservation, compatibility with the
polynomial exponentiation step, and propagation of the symmetry through
iteration, and supplements the algebraic arguments with surface
visualizations illustrating the corresponding polynomial identities. This
material is included for its pedagogical and structural value rather than as
a claim of a new symmetry theorem.

\section{Preliminaries and Numerical Setup}\label{sec:preliminaries}
For $j\ge2$, define
\[
z_0(c)=0,\qquad z_{n+1}(c)=z_n(c)^j+c,
\]
and
\[
M_j=\{c\in\mathbb C:(z_n(c))_{n\ge0}\text{ is bounded}\}.
\]
The escape-time computation uses the standard bailout radius $2$.

\begin{lemma}[Escape criterion]\label{lem:escape}
For every $j\ge2$, $M_j\subseteq\{c:|c|\le2\}$. Moreover, if
$|z_n(c)|>2$ for some $n$, then $|z_k(c)|\to\infty$ as $k\to\infty$.
\end{lemma}

\begin{proof}
If $|c|\le2$ and $|z_n|>2$, then
\[
|z_{n+1}|\ge |z_n|^j-|c|\ge |z_n|^2-2>|z_n|.
\]
Once outside the radius-$2$ disk, the moduli therefore increase and in fact
grow at least geometrically. If $|c|>2$, then $z_1=c$ already lies outside
the bailout disk and the same conclusion follows. Hence parameters with
$|c|>2$ are not in $M_j$.
\end{proof}

\subsection{Escape-time area estimator}\label{sec:computation}
The parameter plane is sampled on $[-2,2]\times[-2,2]$. At every grid point
$c$, the critical orbit is iterated until either $|z_n|>2$ or the prescribed
iteration threshold $K$ is reached. Points that have not escaped by time $K$
are counted as members of the finite-iteration approximation. If the sampling
density is $\rho$ points per unit length, each grid point represents area
$\rho^{-2}$, so the pixel-counting estimate is the number of non-escaping grid
points multiplied by $\rho^{-2}$.

All experiments use the same deterministic Cartesian grid and a
Numba-accelerated implementation with parallel execution. The full reference
implementation is included in Appendix~\ref{app:implementation}; the main text
focuses on the numerical quantities that enter the analysis.

\section{Area Estimates and Two-Parameter Refinement}\label{sec:area}

For each degree $j\in\{2,3,\ldots,20\}$, we compute an escape-time area
estimate on the finest tested grid using $K=1000$ and $\rho=1024$. These
values, reported in Table~\ref{tab:j_vs_area}, form the data set used for the
degree-dependent fit in Section~\ref{sec:scaling}. The remainder of this
section examines how the reported values change when $K$ and $\rho$ are varied
independently.

\begin{table}[htbp]
\centering
\caption{Area approximations for generalized Mandelbrot sets of different degrees $j$, computed using the escape-time algorithm.}
\label{tab:j_vs_area}
\begin{tabular}{cc}
\toprule
Degree $j$ & Area \\
\midrule
2   & 1.5108 \\
3   & 1.7981 \\
4   & 1.9843 \\
5   & 2.1181 \\
6   & 2.2191 \\
7   & 2.2998 \\
8   & 2.3653 \\
9   & 2.4203 \\
10  & 2.4660 \\
11  & 2.5059 \\
12  & 2.5410 \\
13  & 2.5717 \\
14  & 2.5990 \\
15  & 2.6234 \\
16  & 2.6455 \\
17  & 2.6660 \\
18  & 2.6838 \\
19  & 2.7008 \\
20  & 2.7162 \\
\bottomrule
\end{tabular}
\end{table}

\subsection{Numerical Convergence and Error Analysis}
\label{sec:convergence}

The numerical estimation of the Lebesgue measure
\[
\operatorname{Area}(M_j)
=
\iint_{\mathbb C}\chi_{M_j}(c)\,dA
\]
is affected by two distinct approximation parameters. The first is the
finite iteration limit used in the escape-time algorithm. The second is
the finite Cartesian grid used to sample the parameter plane. These two
effects are investigated separately below by varying one parameter while
holding the other fixed.

\subsubsection{Finite-Iteration Approximation}

For a fixed polynomial degree $j\geq 2$, define the finite-iteration set
\[
M_{j,K}
=
\left\{
c\in\mathbb C:
|z_n(c)|\leq 2
\text{ for every }
1\leq n\leq K
\right\},
\]
where
\[
z_0(c)=0,
\qquad
z_{n+1}(c)=z_n(c)^j+c.
\]

If a parameter remains within the escape radius for $K+1$ iterations,
then it necessarily remains within the escape radius for the first $K$
iterations. Therefore,
\[
M_{j,K+1}\subseteq M_{j,K}.
\]

Hence, the finite-iteration sets form a nested decreasing sequence, and
their exact areas satisfy
\[
\operatorname{Area}(M_{j,K+1})
\leq
\operatorname{Area}(M_{j,K}).
\]

Formally, the bounded-orbit Multibrot set can be written as
\[
M_j
=
\bigcap_{K=1}^{\infty}M_{j,K}.
\]

If the exact areas of the sets $M_{j,K}$ were available, continuity of
Lebesgue measure from above would imply
\[
\lim_{K\to\infty}
\operatorname{Area}(M_{j,K})
=
\operatorname{Area}(M_j).
\]

The computations below, however, do not evaluate
$\operatorname{Area}(M_{j,K})$ exactly. They approximate each
finite-iteration set using a discrete grid. Consequently, the numerical
data should be interpreted as empirical approximations to the
finite-iteration areas rather than as a proof that the complete
pixel-counting procedure converges to the true Multibrot area.

\subsubsection{Grid Convention and Finite Precision}

All numerical experiments reported in this section were generated using
a single Numba-accelerated implementation with parallel execution and
fast-math optimization. The parameter plane was sampled over
$[-2,2]\times[-2,2]$ using a deterministic Cartesian grid. For a sampling
density $\rho$, the grid contains $(4\rho)\times(4\rho)$ points, with
coordinates
\[
c_{m,n}
=
\left(-2+\frac{m}{\rho}\right)
+
i\left(-2+\frac{n}{\rho}\right),
\qquad
0\leq m,n<4\rho.
\]
Thus, the lower and left boundaries of the square are included, while
the upper and right boundaries are excluded. An orbit was terminated
immediately once its squared modulus exceeded $4$. The estimated area
was computed as the number of parameters not escaping within $K$
iterations multiplied by the area $\rho^{-2}$ represented by each grid
point.

Because fast-math optimization permits non-strict floating-point
rearrangements, parameters extremely close to the computed boundary may
be sensitive to machine-level evaluation. The results should therefore
be regarded as empirical finite-precision estimates.

\subsubsection{Empirical Convergence with Respect to the Iteration Limit}

To investigate the dependence of the computed areas on the iteration
threshold, the escape-time algorithm was evaluated for
\[
K\in\{25,100,500,1000\},
\]
while holding the spatial sampling density fixed at 512 samples per unit
length.

\vspace{1em}
\begin{center}
\captionof{table}{Empirical area estimates obtained at fixed spatial sampling density while varying the iteration limit.}
\label{tab:iteration_convergence}
\begin{tabular}{ccccc}
\toprule
Degree $j$ & $K=25$ & $K=100$ & $K=500$ & $K=1000$ \\
\midrule
2  & 1.6809 & 1.5474 & 1.5155 & 1.5113 \\
3  & 1.8928 & 1.8182 & 1.8000 & 1.7980 \\
4  & 2.0524 & 1.9997 & 1.9860 & 1.9843 \\
5  & 2.1701 & 2.1305 & 2.1194 & 2.1180 \\
6  & 2.2632 & 2.2292 & 2.2198 & 2.2187 \\
7  & 2.3371 & 2.3086 & 2.3007 & 2.3000 \\
8  & 2.3976 & 2.3723 & 2.3657 & 2.3649 \\
9  & 2.4492 & 2.4272 & 2.4212 & 2.4202 \\
10 & 2.4927 & 2.4724 & 2.4667 & 2.4661 \\
11 & 2.5300 & 2.5119 & 2.5069 & 2.5063 \\
12 & 2.5635 & 2.5469 & 2.5419 & 2.5413 \\
13 & 2.5927 & 2.5781 & 2.5733 & 2.5728 \\
14 & 2.6177 & 2.6037 & 2.5994 & 2.5988 \\
15 & 2.6410 & 2.6276 & 2.6239 & 2.6235 \\
16 & 2.6619 & 2.6499 & 2.6460 & 2.6456 \\
17 & 2.6824 & 2.6700 & 2.6660 & 2.6654 \\
18 & 2.6992 & 2.6878 & 2.6843 & 2.6839 \\
19 & 2.7145 & 2.7043 & 2.7012 & 2.7008 \\
20 & 2.7293 & 2.7195 & 2.7164 & 2.7160 \\
\bottomrule
\end{tabular}
\end{center}
\vspace{1em}

The estimates are monotonically decreasing across the tested
iteration limits, consistent with the nesting relation
\[
M_{j,K+1}\subseteq M_{j,K}.
\]

The largest corrections generally occur between $K=25$ and $K=100$.
For example, for the degree-$10$ Multibrot set,
\[
2.4927
\rightarrow
2.4724
\rightarrow
2.4667
\rightarrow
2.4661,
\]
with successive differences
\[
0.0203,\qquad
0.0057,\qquad
0.0006.
\]

A similar decrease in iteration sensitivity is observed throughout the
tested range of degrees.

\vspace{1em}
\begin{center}
\captionof{table}{Absolute difference between the $K=500$ and $K=1000$ area estimates.}
\label{tab:iteration_difference}
\begin{tabular}{cc}
\toprule
Degree $j$
&
$\left|A_{j,1000}-A_{j,500}\right|$
\\
\midrule
2  & 0.0042 \\
3  & 0.0020 \\
4  & 0.0017 \\
5  & 0.0014 \\
6  & 0.0011 \\
7  & 0.0007 \\
8  & 0.0009 \\
9  & 0.0010 \\
10 & 0.0007 \\
11 & 0.0005 \\
12 & 0.0006 \\
13 & 0.0005 \\
14 & 0.0006 \\
15 & 0.0003 \\
16 & 0.0005 \\
17 & 0.0006 \\
18 & 0.0005 \\
19 & 0.0004 \\
20 & 0.0005 \\
\bottomrule
\end{tabular}
\end{center}
\vspace{1em}

Across all reported degrees, the difference between the $K=500$ and
$K=1000$ estimates does not exceed approximately
\[
4.2\times10^{-3}
\]
(and remains below $1.0\times10^{-3}$ for degrees $j\geq 7$).

These results indicate decreasing sensitivity to the iteration threshold
over the tested range. This stabilization does not provide a rigorous
error bound relative to the true area. A small difference between two
finite-iteration calculations only shows that the two numerical
approximations are close to one another.

\subsubsection{Grid-Refinement Study}
\label{sec:grid_refinement}

To study the effect of spatial resolution separately from the iteration
limit, the iteration threshold was fixed at
\[
K=1000,
\]
and the sampling density was varied over
\[
\rho\in\{128,256,512,1024\}
\]
samples per unit length.

Because the computations were performed over the square
\[
[-2,2]\times[-2,2],
\]
these densities correspond respectively to the grids
\[
512\times512,\qquad
1024\times1024,\qquad
2048\times2048,\qquad
4096\times4096.
\]

\begin{table}[htbp]
\centering
\caption{Grid-refinement study at fixed iteration limit $K=1000$.
The sampling density is measured in grid points per unit length.}
\label{tab:grid_refinement}
\begin{tabular}{ccccc}
\toprule
Degree $j$
&
$\rho=128$
&
$\rho=256$
&
$\rho=512$
&
$\rho=1024$
\\
\midrule
2  & 1.515686 & 1.512619 & 1.511349 & 1.510802 \\
3  & 1.797668 & 1.797440 & 1.798008 & 1.798062 \\
4  & 1.984985 & 1.984207 & 1.984333 & 1.984312 \\
5  & 2.117249 & 2.118057 & 2.117985 & 2.118096 \\
6  & 2.217529 & 2.217438 & 2.218693 & 2.219094 \\
7  & 2.299866 & 2.300064 & 2.300007 & 2.299790 \\
8  & 2.363953 & 2.364517 & 2.364861 & 2.365273 \\
9  & 2.418030 & 2.419571 & 2.420200 & 2.420255 \\
10 & 2.464111 & 2.464981 & 2.466053 & 2.466035 \\
11 & 2.506165 & 2.507004 & 2.506344 & 2.505893 \\
12 & 2.544678 & 2.541718 & 2.541306 & 2.540999 \\
13 & 2.564514 & 2.570389 & 2.572788 & 2.571725 \\
14 & 2.597961 & 2.599014 & 2.598835 & 2.598992 \\
15 & 2.621155 & 2.624588 & 2.623539 & 2.623408 \\
16 & 2.644043 & 2.646240 & 2.645580 & 2.645460 \\
17 & 2.662659 & 2.664932 & 2.665440 & 2.666024 \\
18 & 2.683655 & 2.683533 & 2.683880 & 2.683821 \\
19 & 2.699646 & 2.700211 & 2.700787 & 2.700835 \\
20 & 2.714478 & 2.715378 & 2.715992 & 2.716249 \\
\bottomrule
\end{tabular}
\end{table}

The estimates exhibit relatively small changes at the finest spatial
refinements considered, although the successive corrections are neither
monotone nor uniformly decreasing across all degrees. For example, for $j=10$,
\[
2.464111
\rightarrow
2.464981
\rightarrow
2.466053
\rightarrow
2.466035,
\]
with successive absolute changes
\[
0.000870,\qquad
0.001072,\qquad
0.000018.
\]

The refinement behavior is not strictly monotone for many of the tested
degrees. This is not inconsistent with the numerical method, because
successive Cartesian grids may intersect thin boundary structures
differently. Unlike the finite-iteration sets, the sampled sets produced
at different grid densities do not form a nested sequence. Grid
refinement therefore need not produce a one-sided sequence of area
estimates.

To summarize the final refinement step,
Table~\ref{tab:grid_difference} reports the absolute difference between
the densities $\rho=512$ and $\rho=1024$.

\begin{table}[htbp]
\centering
\caption{Absolute difference between the two finest grid resolutions at
fixed iteration limit $K=1000$.}
\label{tab:grid_difference}
\begin{tabular}{cc}
\toprule
Degree $j$
&
$\left|A_{j,1024}-A_{j,512}\right|$
\\
\midrule
2  & 0.000547 \\
3  & 0.000054 \\
4  & 0.000021 \\
5  & 0.000112 \\
6  & 0.000401 \\
7  & 0.000216 \\
8  & 0.000412 \\
9  & 0.000054 \\
10 & 0.000018 \\
11 & 0.000451 \\
12 & 0.000307 \\
13 & 0.001063 \\
14 & 0.000157 \\
15 & 0.000131 \\
16 & 0.000120 \\
17 & 0.000585 \\
18 & 0.000059 \\
19 & 0.000049 \\
20 & 0.000257 \\
\bottomrule
\end{tabular}
\end{table}

Across the tested degrees, the final grid-refinement correction ranges
from approximately
\[
1.8\times10^{-5}
\]
to
\[
1.06\times10^{-3}.
\]

Thus, at $K=1000$, increasing the grid from
$2048\times2048$ to $4096\times4096$ changes the reported area estimates
by no more than approximately $1.1\times10^{-3}$ over the tested degree
range.

This provides evidence that the computations become less sensitive to
spatial resolution at the finest grids considered. It does not establish
a rigorous discretization error bound or prove convergence of the
pixel-counting estimates to the true Multibrot area.

\subsubsection{Boundary Measure and Interpretation of the Estimates}

The boundary of a Multibrot set is highly irregular, and a finite
Cartesian grid may sample fine boundary structures differently as the
resolution changes. Shishikura~\cite{shishikura1998} proved that the
boundary of the classical Mandelbrot set has Hausdorff dimension two.
As noted by Bittner et al.~\cite{bittner2014}, this leaves open the
possibility that the boundary may possess positive two-dimensional
Lebesgue measure. Consequently, it is not presently justified to assume
that direct pixel-counting estimates necessarily converge to the true
area merely by increasing the lattice resolution.

Andreadis and Karakasidis~\cite{andreadis2015numerical} considered this boundary
sensitivity through a boundary-scanning method that distinguishes
interior, exterior, and boundary lattice points and applies Pick's
theorem to construct an alternative numerical area approximation. They
also studied the dependence of the projected Mandelbrot set on both the
lattice resolution and the maximum iteration count. The present work
does not implement their boundary-scanning procedure; instead, it retains
direct grid counting and empirically studies the sensitivity of the
computed estimates to both numerical parameters.

Accordingly, the iteration-limit and grid-refinement experiments support
a limited conclusion: over the parameter ranges tested, the area
estimates exhibit numerical stabilization with respect to both the
iteration threshold and the grid resolution. They do not establish that
the observed limiting values equal the exact Lebesgue areas of the
corresponding Multibrot sets and should instead be interpreted as
empirical finite-resolution, finite-precision estimates.

The numerical estimates are complemented in
Section~\ref{sec:laurent} by analytic upper bounds obtained from
Gronwall's Area Theorem. These bounds provide an independent theoretical
comparison, but they do not by themselves determine the numerical error
of the grid-based estimates.

\section{Degree--Area Scaling}\label{sec:scaling}
The finest tested estimates use $K=1000$ and $\rho=1024$. Across the computed
degrees, these values increase toward $\pi$, consistent with the known
large-degree geometric limit. Figure~\ref{fig:area_convergence} shows the
resulting trend.

\begin{figure}[htbp]
    \centering
    \includegraphics[width=0.92\textwidth]{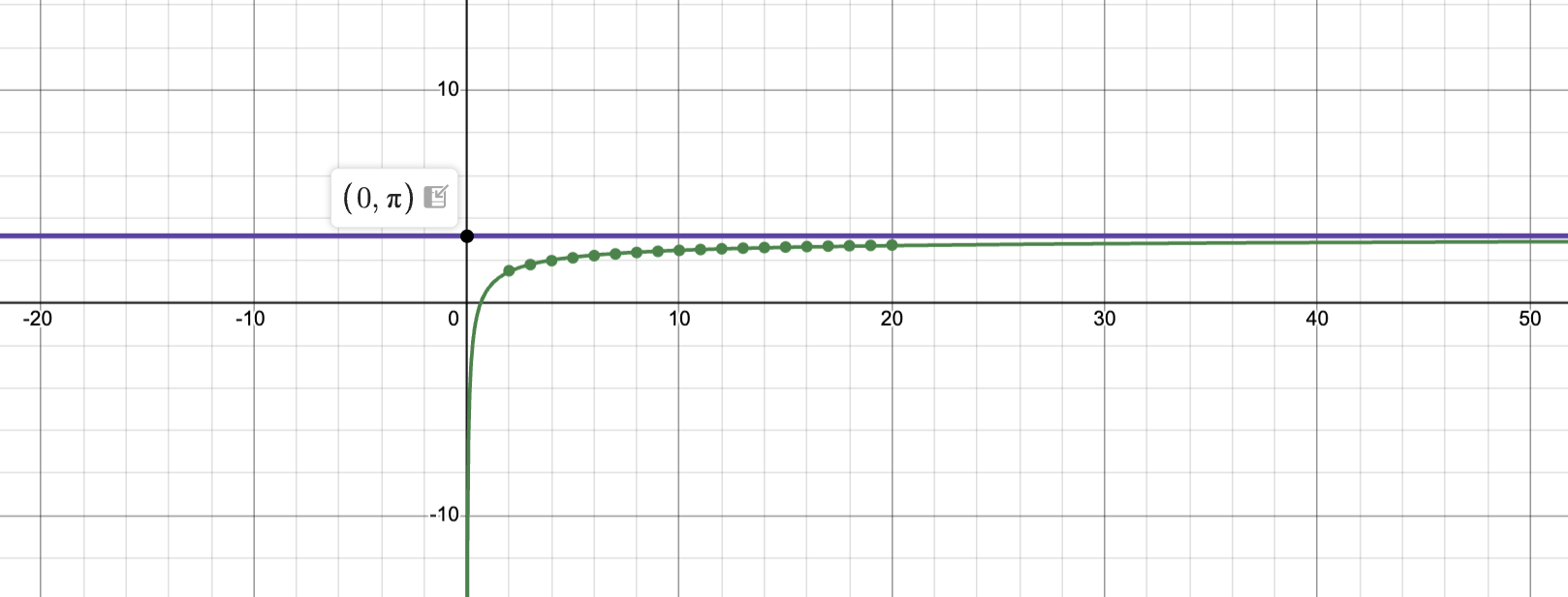}
    \caption{Grid-based area estimates as a function of polynomial degree.
    The curve is descriptive rather than a proof of convergence of the
    pixel-counting scheme to exact Lebesgue area.}
    \label{fig:area_convergence}
\end{figure}

To summarize the finite-degree dependence, we fit
\[
C(j)=\pi-\beta j^{-\alpha}.
\]
For the reported data, the fitted model is
\begin{equation}\label{eq:empirical_fit}
C(j)=\pi-\frac{2.49216}{j^{0.57103}},
\end{equation}
with $R^2=0.996$. Equivalently, the computed area deficit
$\Delta(j)=\pi-C(j)$ is described over the fitted range by a power law
$\Delta(j)\approx\beta j^{-\alpha}$.

Equation~\eqref{eq:empirical_fit} is an empirical summary of the computed
values, not a theorem about the exact areas of $M_j$. Its role is to provide a
quantitative target that can be compared with the independent conformal upper
bounds in Section~\ref{sec:laurent}. Additional high-degree renderings and
geometric observations are collected in Appendix~\ref{app:visuals}.

\section{Conformal Upper Bounds}\label{sec:laurent}
The grid-refinement study measures numerical sensitivity but does not provide
a rigorous discretization error bound. We therefore compare the escape-time
estimates with an independent analytic construction based on the exterior
Riemann map and Gronwall's area theorem. The resulting quantities are rigorous
upper bounds; they should be interpreted as an external consistency check,
not as a proof that the grid estimates have a known absolute error.

\subsection{Exterior map and Laurent coefficients}
For each degree $j\ge2$, let
\[
\Phi_j:\widehat{\mathbb C}\setminus M_j\longrightarrow\{w:|w|>1\}
\]
be the normalized exterior conformal map, with
\[
\frac{\Phi_j(c)}{c}\to1\qquad(c\to\infty),
\]
and write $\Psi_j=\Phi_j^{-1}$. Near infinity,
\begin{equation}\label{eq:laurent_series}
\Psi_j(w)=w+b_{j,0}+\sum_{m=1}^{\infty}\frac{b_{j,m}}{w^m}.
\end{equation}
The coefficients $b_{j,m}$ encode geometric information about the boundary
and are the inputs to Gronwall's formula.

We generalized a quadratic Laurent-coefficient implementation to arbitrary
Multibrot degree and checked the output against known coefficient structure.
For $j\ge3$,
\begin{equation}\label{eq:multibrot_zero_coefficients}
b_{j,m}=0\qquad\text{whenever}\qquad (j-1)\nmid(m+1),
\end{equation}
and $b_{j,0}=0$. The leading nonzero coefficient satisfies
\begin{equation}\label{eq:leading_coefficient}
b_{j,j-2}=-\frac1j.
\end{equation}
For example, when $j=4$, the implementation reproduces
$b_{4,2}=-1/4$ and $b_{4,5}=-1/32$ together with the expected sparsity
pattern. These exact checks test the generalized coefficient computation
before the coefficients are used in the area calculation. A longer
coefficient table and boundary reconstructions are provided in
Appendix~\ref{app:conformal_checks}.

\subsection{Gronwall area bounds}
For a normalized exterior map of the form in
Eq.~\eqref{eq:laurent_series}, Gronwall's area formula gives
\begin{equation}\label{eq:gronwall}
\operatorname{Area}(M_j)
=\pi\left(1-\sum_{m=1}^{\infty}m|b_{j,m}|^2\right).
\end{equation}
Truncating after $M$ coefficients yields
\begin{equation}\label{eq:gronwall_truncated}
U_{j,M}
=\pi\left(1-\sum_{m=1}^{M}m|b_{j,m}|^2\right).
\end{equation}
Because the omitted terms are nonnegative,
\begin{equation}\label{eq:gronwall_bound}
\operatorname{Area}(M_j)\le U_{j,M}.
\end{equation}
Thus every finite truncation provides a rigorous upper bound. We evaluate
Eq.~\eqref{eq:gronwall_truncated} with $M=1000$ for
$j\in\{3,4,\ldots,20\}$.

\subsection{Comparison with grid estimates}

Figure~\ref{fig:curve_validation} overlays the $M=1000$ Gronwall upper
bounds with the $K=1000$, $\rho=1024$ escape-time estimates.

\begin{figure}[htbp]
    \centering
    \includegraphics[width=0.86\textwidth]{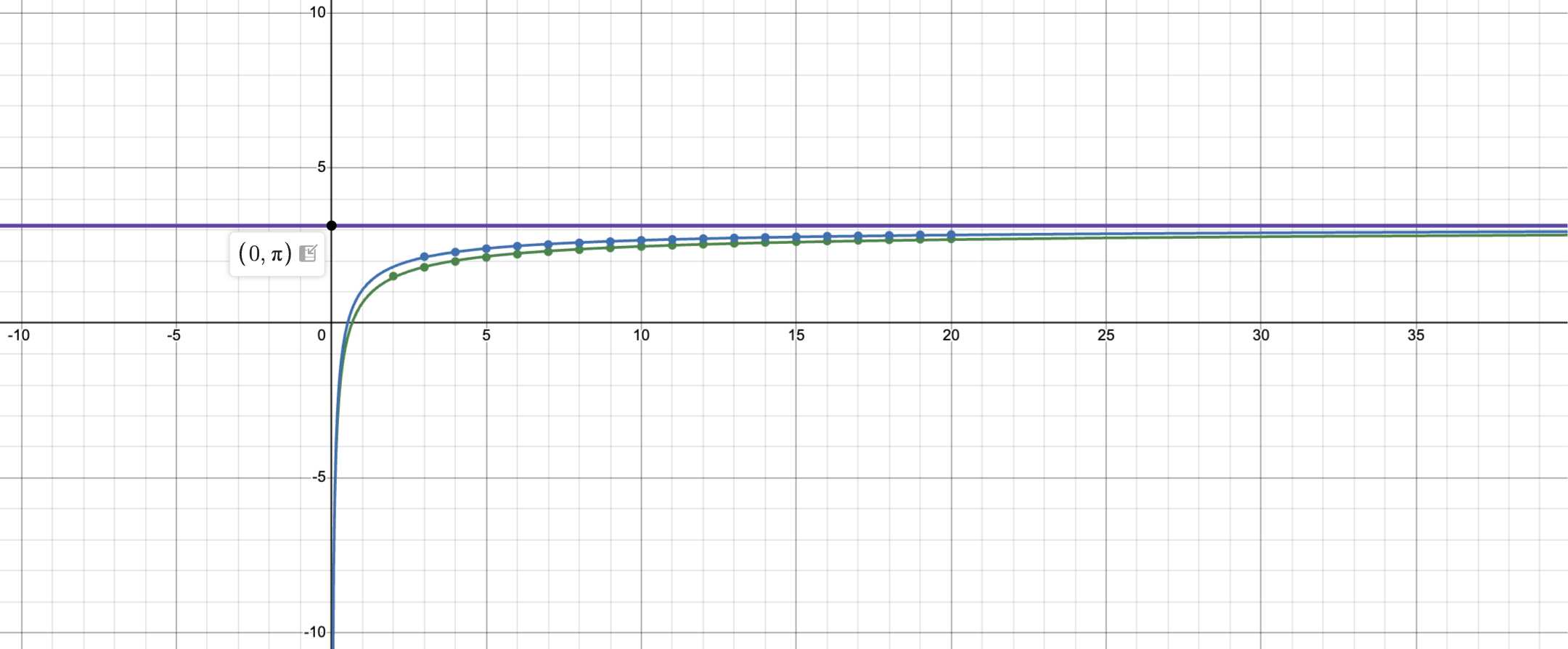}
    \caption{Comparison of grid-based area estimates with $M=1000$ Gronwall
    upper bounds. Across the tested degrees, the numerical estimates remain
    below the independent conformal bounds. Both sequences are well described
    over the computed range by degree-dependent power-law deficits from
    $\pi$, although with different fitted parameters.}
    \label{fig:curve_validation}
\end{figure}

For every tested degree, the grid estimate lies below the corresponding
Gronwall upper bound, as required for consistency with
Eq.~\eqref{eq:gronwall_bound} if the pixel estimate is close to the true
area. The sequence of $M=1000$ Gronwall bounds is itself well described over
the tested range by the empirical fit
\begin{equation}
    U_{1000}(j)
    \approx
    \pi-\frac{2.07405}{j^{0.643427}},
    \label{eq:gronwall_fit}
\end{equation}
with $R^2=0.9976$.

By comparison, the independently obtained grid-based estimates satisfy
\begin{equation}
    C(j)
    \approx
    \pi-\frac{2.49216}{j^{0.57103}},
\end{equation}
with $R^2=0.996$. Thus, although the fitted exponents and coefficients are
not identical, both independently generated sequences are well described
over the computed degree range by the same functional family
\[
    \pi-\beta j^{-\alpha}.
\]

This agreement is notable because the two sequences arise from distinct
constructions: one from finite-grid orbit classification and the other from
Laurent coefficients of the exterior conformal map. The fitted curve in
Eq.~\eqref{eq:gronwall_fit} is an empirical description of the
finite-truncation Gronwall bounds and is not itself asserted to be a rigorous
upper bound.

The comparison should nevertheless be interpreted conservatively. An upper
bound does not determine the error of a lower numerical estimate, and the
observed agreement does not prove that the grid sequence converges to the
exact area. Rather, the Gronwall bounds provide an independent conformal
reference for both the scale and the observed degree dependence of the
escape-time estimates.

\section{Conclusion}
This paper develops a three-stage quantitative study of Multibrot area across
polynomial degree. Escape-time pixel counting supplies the primary numerical
estimates. A two-parameter refinement study then separates the theoretically
controlled finite-iteration approximation from the empirically assessed grid
discretization. Finally, Laurent coefficients and Gronwall's area theorem
supply rigorous conformal upper bounds for an independent comparison.

At the finest tested parameters, the final grid refinement changes the
reported areas by at most approximately $1.1\times10^{-3}$ across the tested
degrees. The resulting values are well summarized by the empirical law
$C(j)=\pi-\beta j^{-\alpha}$ with $R^2=0.996$, while the $M=1000$ conformal
upper bounds exhibit a closely related degree dependence with $R^2=0.9976$.
These fits quantify the observed finite-degree trend but are not claimed as
asymptotic theorems for the exact areas.

The main unresolved question is therefore analytic rather than computational:
whether the observed power-law behavior of the area deficit can be derived
from the geometry or exterior conformal map of $M_j$. Further numerical work
could refine the discretization study or compare alternative area estimators,
while further conformal analysis could investigate whether the coefficient
structure explains the fitted exponent. The symmetry derivations retained in
Appendix~\ref{app:symmetry} provide complementary geometric intuition, but the
central result of the paper is the quantitative connection between numerical
area estimates, refinement behavior, and conformal upper bounds.

\section*{Acknowledgments}
I sincerely thank my mentor, Jason Liang, for his guidance and support during the early stages of this research. His mentorship helped lay the foundation for this project and sparked my interest in the mathematical study of generalized Mandelbrot sets. I would further like to thank Majid Mahzoon for his guidance on the framing of this work in its early stages, and Rajkishore Barik and Joseph Laurienzo for encouraging me to engage more thoroughly with the relevant literature. Finally, I would like to thank Donald Snedden, my math teacher, who first suggested exploring Laurent series, and Eric Hallman, who encouraged me to conduct a rigorous numerical analysis convergence study. Together, our discussions on this topic contributed greatly to the development of this paper.

\appendix

\section{Symmetry and Geometric Structure}\label{app:symmetry}
\label{sec:symmetry}

We now formally investigate the reflectional and rotational symmetries inherent to generalized Mandelbrot sets. These symmetries are governed by the dihedral group $D_{j-1}$, the symmetry group of a regular $(j-1)$-gon. In particular, the generalized Multibrot sets exhibit both $(j-1)$-fold rotational symmetry and $(j-1)$ reflection symmetries. Figures~\ref{fig:roots_unity} for degrees $j\in\{4,5,6\}$ provide visual illustrations of these geometric properties.

\begin{figure}[htbp]
    \centering
    \begin{minipage}{0.32\textwidth}
        \centering
        \includegraphics[width=\textwidth]{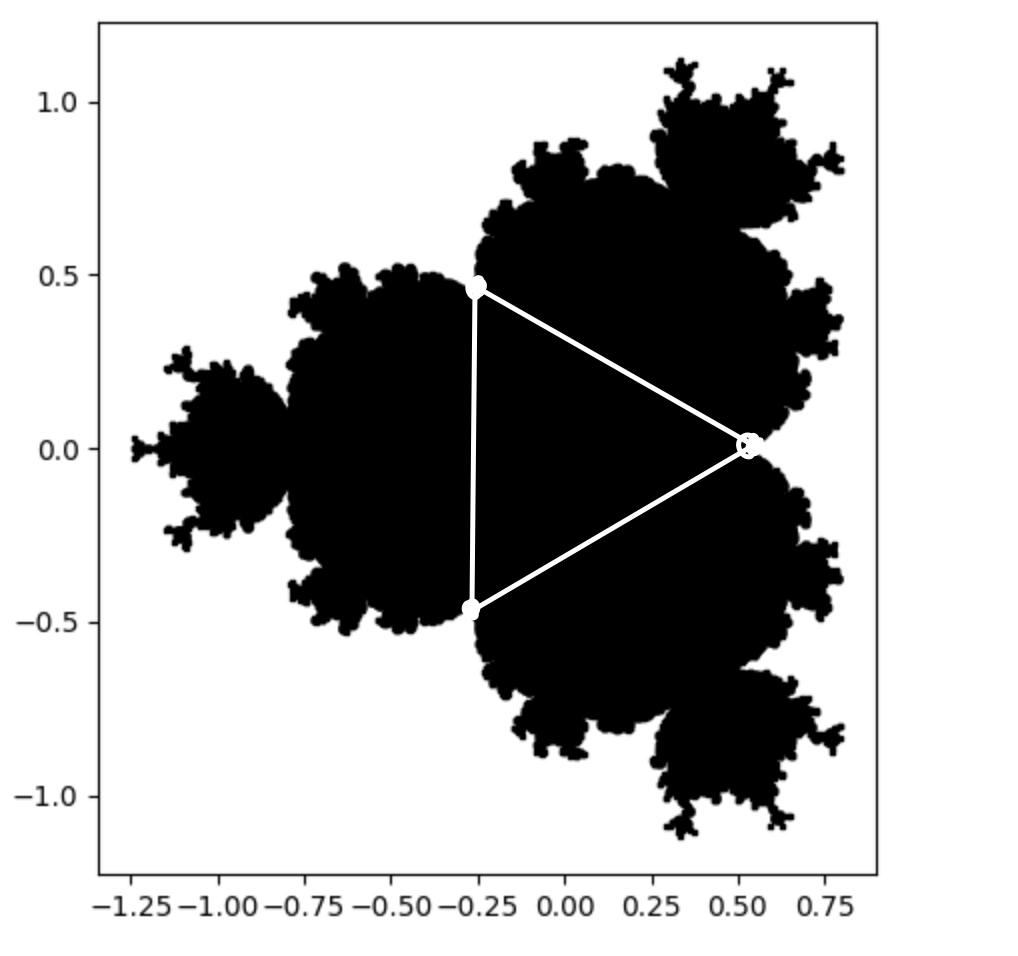}
        \caption*{$f(z) = z^4 + c$ \\ (Triangle)}
    \end{minipage}\hfill
    \begin{minipage}{0.32\textwidth}
        \centering
        \includegraphics[width=\textwidth]{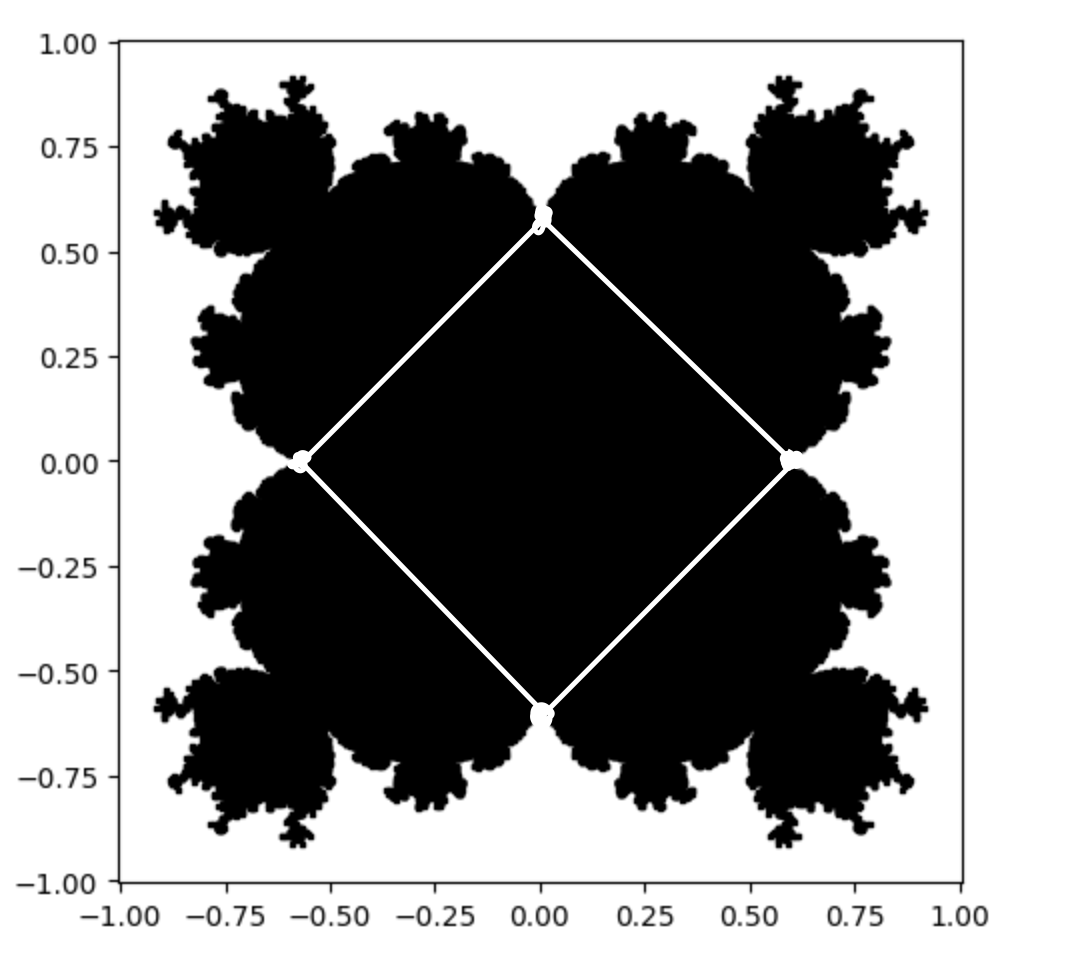}
        \caption*{$f(z) = z^5 + c$ \\ (Square)}
    \end{minipage}\hfill
    \begin{minipage}{0.32\textwidth}
        \centering
        \includegraphics[width=\textwidth]{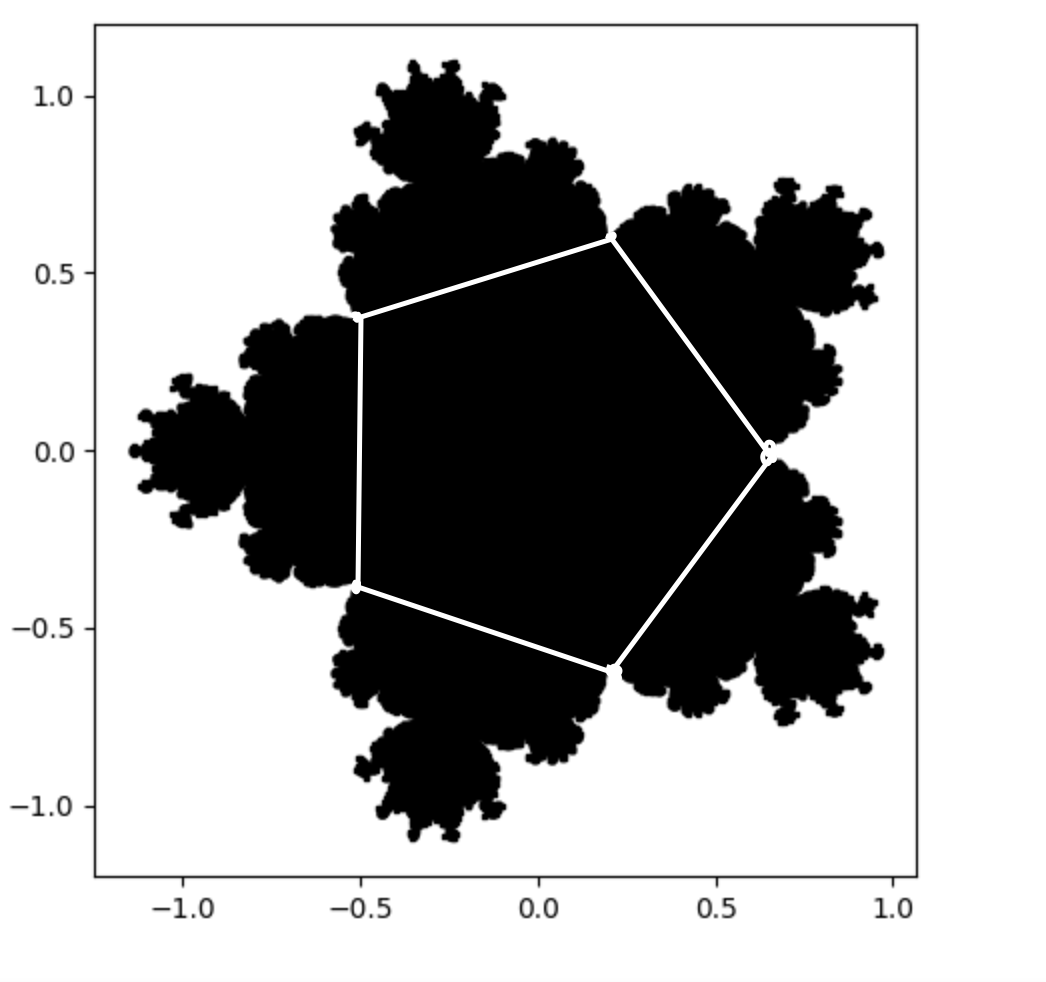}
        \caption*{$f(z) = z^6 + c$ \\ (Pentagon)}
    \end{minipage}
    \caption{Roots of unity structures for $f(z)=z^j+c$ across $j=4,5,6$. Regular $(j-1)$-gons formed by roots of unity placed on the complex unit circle naturally bound the central core geometry of each generalized set.}
    \label{fig:roots_unity}
\end{figure}

More precisely, for a generalized set of degree $j$, the vertices of the inscribed polygon correspond to scaled $(j-1)$-st roots of unity. Introducing a uniform scaling factor $k > 0$, constructing the explicit lines of symmetry reduces to evaluating the perpendicular bisectors originating from these dilated vertices through the origin.

Figure~\ref{fig:symmetry_lines} illustrates the resulting symmetry frameworks and coordinate conversions for $j \in \{4,5,6\}$.

\begin{figure}[htbp]
    \centering
    \begin{minipage}{0.32\textwidth}
        \centering
        \includegraphics[width=\textwidth]{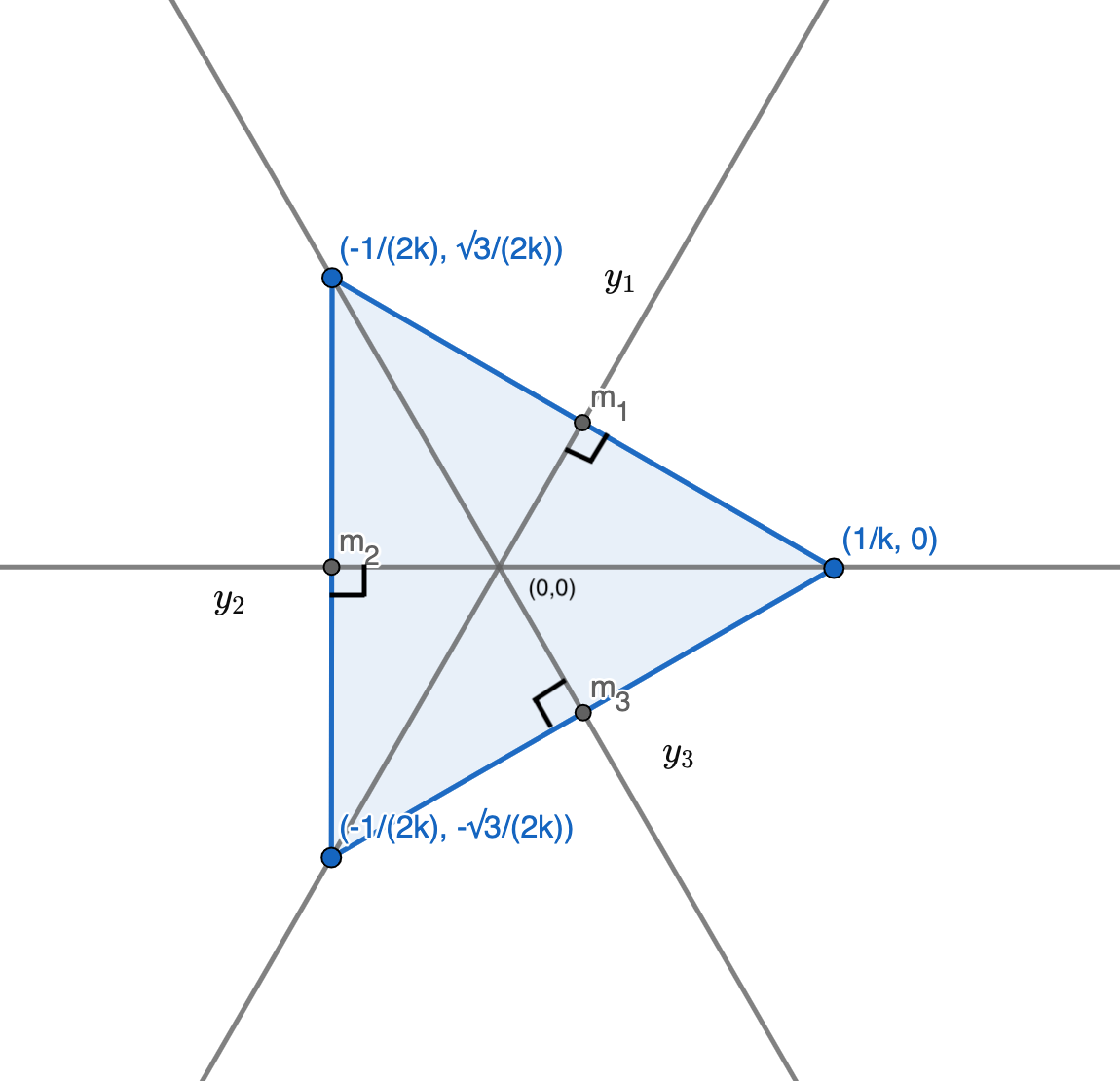}
        \caption*{$f(z) = z^4 + c$ \\ Cube roots of unity}
    \end{minipage}\hfill
    \begin{minipage}{0.32\textwidth}
        \centering
        \includegraphics[width=\textwidth]{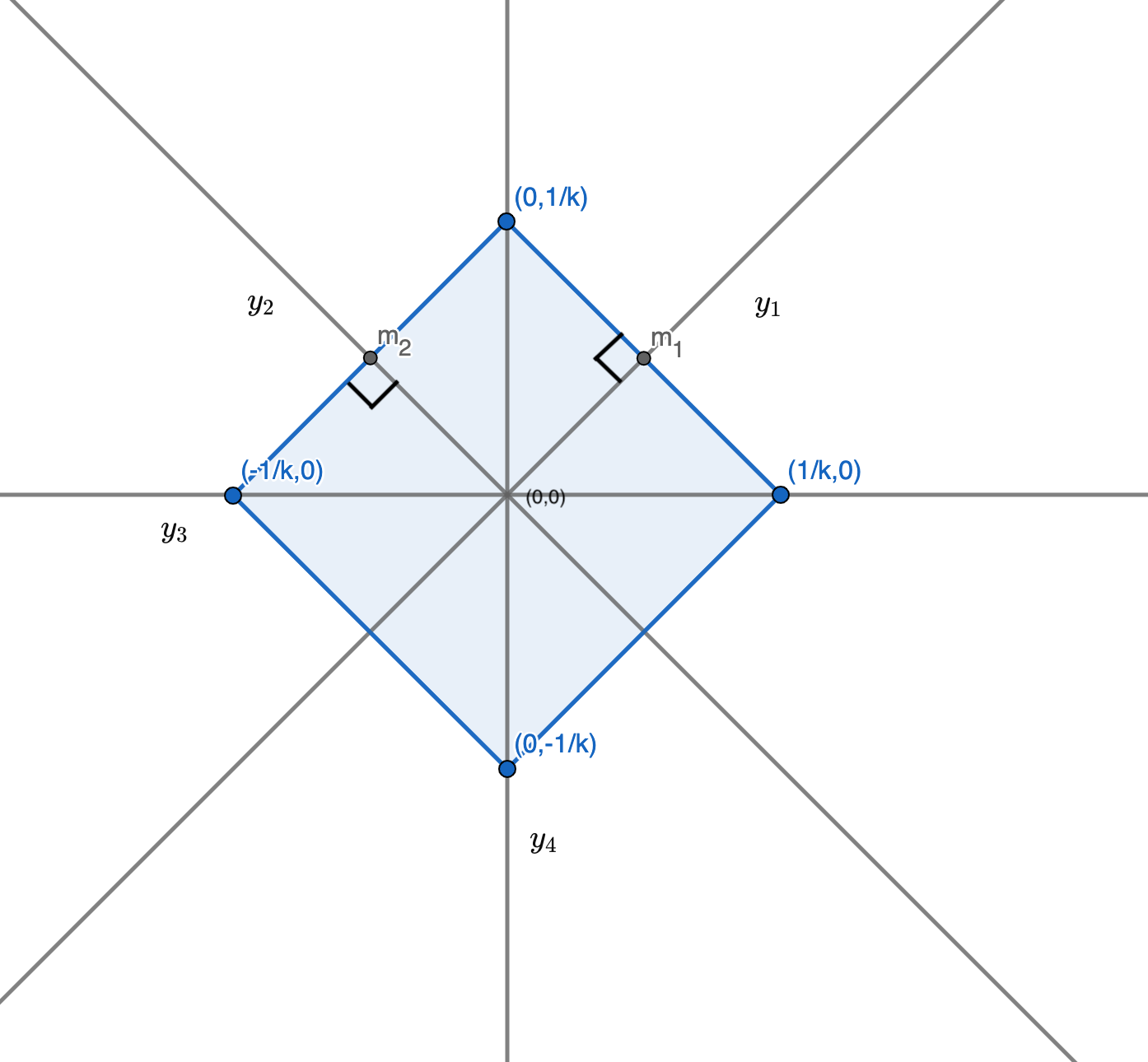}
        \caption*{$f(z) = z^5 + c$ \\ 4th roots of unity}
    \end{minipage}\hfill
    \begin{minipage}{0.32\textwidth}
        \centering
        \includegraphics[width=\textwidth]{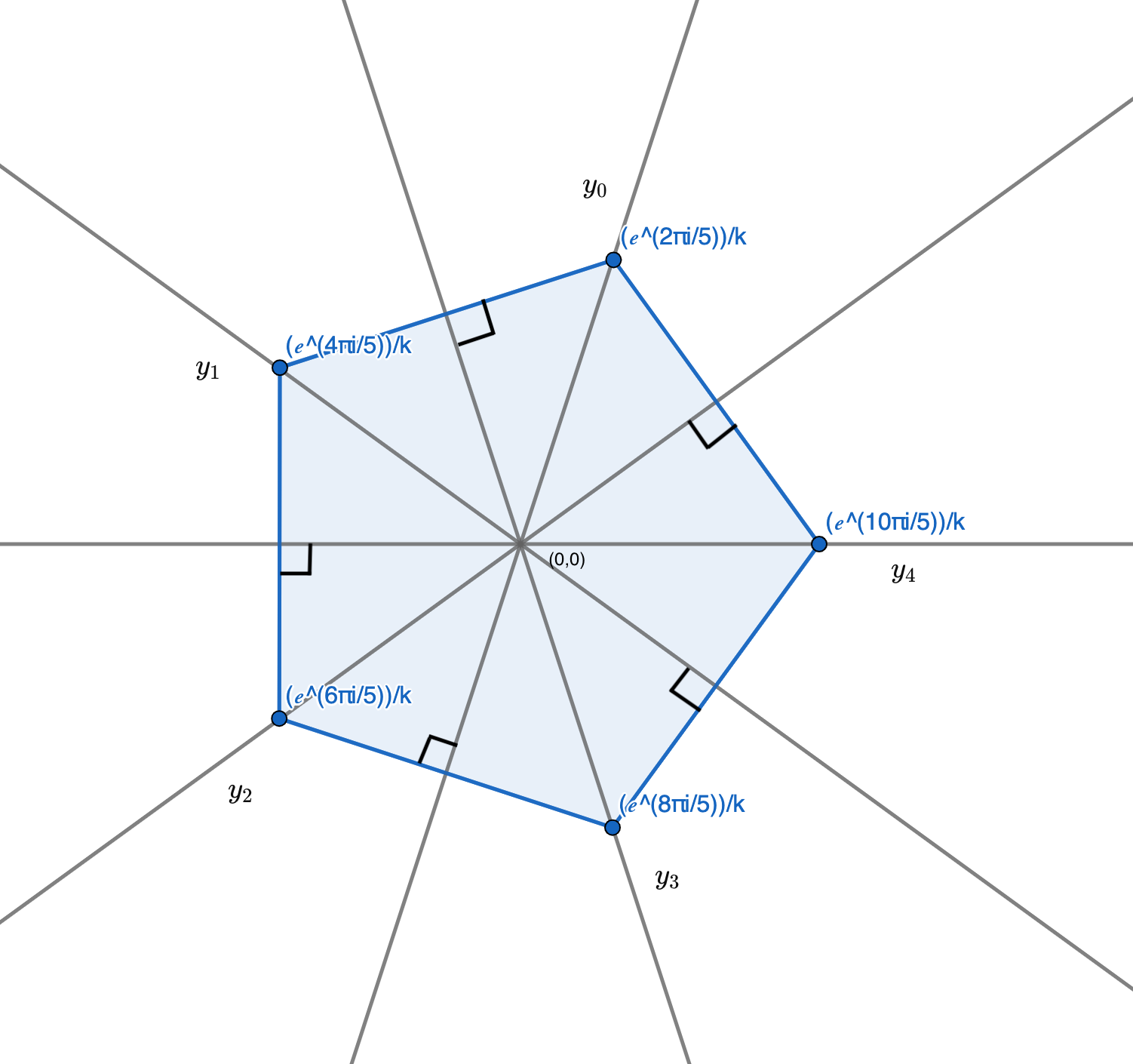}
        \caption*{$f(z) = z^6 + c$ \\ 5th roots of unity}
    \end{minipage}
    \caption{Symmetry line derivations for generalized Mandelbrot sets. Symmetry axes were constructed by converting the roots of unity $z^n=1$ ($n=3,4,5$) into Cartesian coordinates and calculating their perpendicular bisectors through the origin.}
    \label{fig:symmetry_lines}
\end{figure}

To demonstrate this explicit line derivation, consider $f(z) = z^4 + c$. The scaled roots of unity satisfy $(kz)^3 = 1 = e^{2\pi i}$, yielding:
\begin{align*}
    kz_0 &= 1 \implies z_0 = \left(\frac{1}{k}, 0\right) \\
    kz_1 &= e^{2\pi i/3} = -\frac{1}{2} + \frac{\sqrt{3}}{2}i \implies z_1 = \left(-\frac{1}{2k}, \frac{\sqrt{3}}{2k}\right) \\
    kz_2 &= e^{4\pi i/3} = -\frac{1}{2} - \frac{\sqrt{3}}{2}i \implies z_2 = \left(-\frac{1}{2k}, -\frac{\sqrt{3}}{2k}\right)
\end{align*}
Evaluating the lines connecting each vertex through the origin $(0,0)$ via slope-intercept form gives the three principal axes of reflection:
\begin{equation*}
    y_1 = \sqrt{3}x, \quad y_2 = 0, \quad y_3 = -\sqrt{3}x.
\end{equation*}

\subsection{Formal Proof of Invariance}

To rigorously prove reflectional and rotational invariance across generalized degrees, we define two transformation operators on $\mathbb{C} \cong \mathbb{R}^2$:
\begin{enumerate}
    \item \textbf{Reflection Operator $R_m(x,y)$:} Reflection across a line through the origin with slope $m$:
    \begin{equation*}
        R_m(x,y) = \left( \frac{(1-m^2)x + 2my}{1+m^2}, \, \frac{(m^2-1)y + 2mx}{1+m^2} \right).
    \end{equation*}
    \item \textbf{Rotation Operator $r_k(c)$:} Rotation by discrete multiples of $\frac{2\pi}{j-1}$:
    \begin{equation*}
        r_k(c) = c \cdot e^{\frac{2\pi i k}{j-1}}, \quad k \in \{0, 1, \dots, j-2\}.
    \end{equation*}
\end{enumerate}

The following technical lemmas establish the required algebraic preservation properties.

\begin{lemma}\label{lem:norm_preservation}
Let $T \in \{R_m, r_k\}$ be any reflection or rotation operator corresponding to degree $j \in \{4,5,6\}$. Then $T$ preserves modulus:
\begin{equation*}
    |T(z)| = |z| \quad \forall z \in \mathbb{C}.
\end{equation*}
\end{lemma}

\begin{lemma}\label{lem:power_commutation}
For each transformation $T$ specified in Lemma~\ref{lem:norm_preservation}, $T$ commutes with the polynomial exponentiation step:
\begin{equation*}
    T(z^j) = (T(z))^j \quad \forall z \in \mathbb{C}.
\end{equation*}
\end{lemma}

\begin{lemma}\label{lem:iterate_commutation}
Let $z_{n,c}$ denote the $n$-th iterate of $f_c(z) = z^j + c$ starting from $z_{0,c} = 0$. For any transformation $T$ satisfying Lemmas~\ref{lem:norm_preservation} and \ref{lem:power_commutation}:
\begin{equation*}
    z_{n, T(c)} = T(z_{n,c}) \quad \forall n \ge 0.
\end{equation*}
\end{lemma}

\begin{proof}
We proceed by induction on the iteration depth $n$.
\begin{itemize}
    \item \textbf{Base Case ($n=0$):} By definition, $z_{0, T(c)} = 0$. Since $T(0) = 0$, $z_{0, T(c)} = T(z_{0,c})$ holds trivially.
    \item \textbf{Inductive Step:} Assume $z_{N, T(c)} = T(z_{N,c})$ holds for some $N \ge 0$. Applying the recurrence relation and Lemma~\ref{lem:power_commutation}:
    \begin{align*}
        z_{N+1, T(c)} &= (z_{N, T(c)})^j + T(c) \\
                      &= (T(z_{N,c}))^j + T(c) \quad \text{(by inductive hypothesis)} \\
                      &= T(z_{N,c}^j) + T(c) \quad \text{(by Lemma~\ref{lem:power_commutation})} \\
                      &= T(z_{N,c}^j + c) \quad \text{(by linearity of $T$)} \\
                      &= T(z_{N+1,c}).
    \end{align*}
\end{itemize}
By induction, $z_{n, T(c)} = T(z_{n,c})$ for all $n \in \mathbb{N}_0$.
\end{proof}

\subsection{Proof of Lemma Properties for $j=6$}

To demonstrate the concrete algebraic verification of Lemmas~\ref{lem:norm_preservation} and \ref{lem:power_commutation}, we present the proof for the most computationally intensive reflection case: $j=6$ across the line $y = \tan\left(\frac{2\pi}{5}\right)x$. The full derivations for all remaining reflectional cases ($j=4, 5$) and rotational symmetry cases are provided in the Appendix and archived in the project GitHub repository. 

\begin{proof}[Proof of Lemma~\ref{lem:norm_preservation} for $j=6$]
Let $z = a + bi \cong (a,b)$ and let $m = \tan\left(\frac{2\pi}{5}\right)$. Using $1 + m^2 = \sec^2\left(\frac{2\pi}{5}\right)$, the reflection operator simplifies to:
\begin{equation*}
    R(a,b) = \left( \frac{(1-m^2)a + 2mb}{\sec^2\left(\frac{2\pi}{5}\right)}, \, \frac{(m^2-1)b + 2ma}{\sec^2\left(\frac{2\pi}{5}\right)} \right).
\end{equation*}
Evaluating $|R(z)|^2$:
\begin{align*}
    |R(a,b)|^2 &= \frac{\left[(1-m^2)a + 2mb\right]^2 + \left[(m^2-1)b + 2ma\right]^2}{\sec^4\left(\frac{2\pi}{5}\right)} \\
    &= \frac{(1-m^2)^2 a^2 + 4m(1-m^2)ab + 4m^2 b^2 + (m^2-1)^2 b^2 + 4m(m^2-1)ab + 4m^2 a^2}{(1+m^2)^2} \\
    &= \frac{(1 + 2m^2 + m^4)a^2 + (1 + 2m^2 + m^4)b^2}{(1+m^2)^2} \\
    &= \frac{(1+m^2)^2 (a^2 + b^2)}{(1+m^2)^2} = a^2 + b^2 = |z|^2.
\end{align*}
Taking the square root yields $|R(z)| = |z|$.
\end{proof}

\begin{proof}[Proof of Lemma~\ref{lem:power_commutation} for $j=6$]
Expanding $z^6 = (a+bi)^6$ via the Binomial Theorem yields:
\begin{equation*}
    z^6 = (a^6 - 15a^4b^2 + 15a^2b^4 - b^6) + i(6a^5b - 20a^3b^3 + 6ab^5).
\end{equation*}
Applying the reflection operator directly to $z^6$ gives:
\begin{align*}
    R(z^6) &= \frac{(1-m^2)(a^6 - 15a^4b^2 + 15a^2b^4 - b^6) + 2m(6a^5b - 20a^3b^3 + 6ab^5)}{\sec^2\left(\frac{2\pi}{5}\right)} \\
    &\quad + i \left[ \frac{(m^2-1)(6a^5b - 20a^3b^3 + 6ab^5) + 2m(a^6 - 15a^4b^2 + 15a^2b^4 - b^6)}{\sec^2\left(\frac{2\pi}{5}\right)} \right].
\end{align*}
Conversely, expanding $(R(z))^6$ symbolically confirms exact agreement term-by-term:
\begin{equation*}
    (R(z))^6 = R(z^6).
\end{equation*}
\end{proof}

While the symbolic expansion involves extensive algebra, Figure~\ref{fig:surface_plots} provides numerical and visual verification of this identity across the complex domain.

\begin{figure}[p] % [p] forces it onto its own full page if needed to prevent float overflow
    \centering
    % Maximum safe margin bleed (15% into left/right margins)
    \hspace*{-0.15\textwidth}%
    \begin{minipage}{0.62\textwidth}
        \centering
        \includegraphics[width=\linewidth, height=13.8cm, keepaspectratio=false]{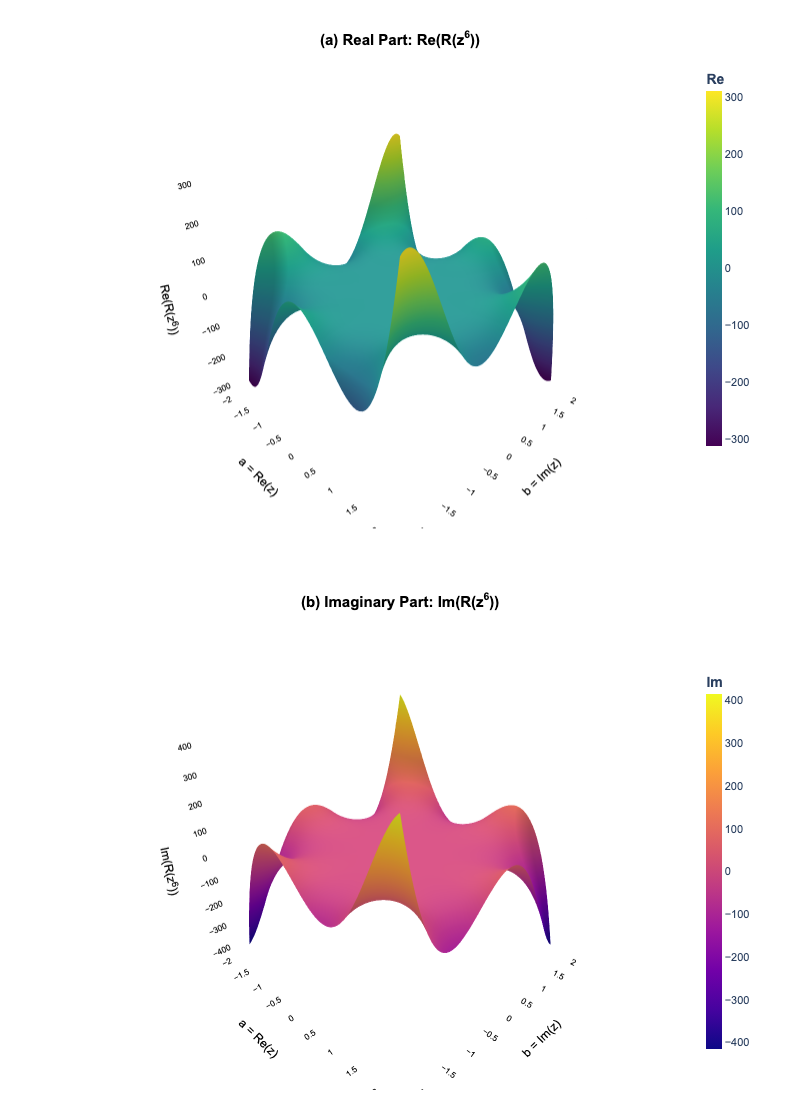}
        \caption*{$R(z^6)$}
    \end{minipage}%
    \hfill%
    \begin{minipage}{0.62\textwidth}
        \centering
        \includegraphics[width=\linewidth, height=13.8cm, keepaspectratio=false]{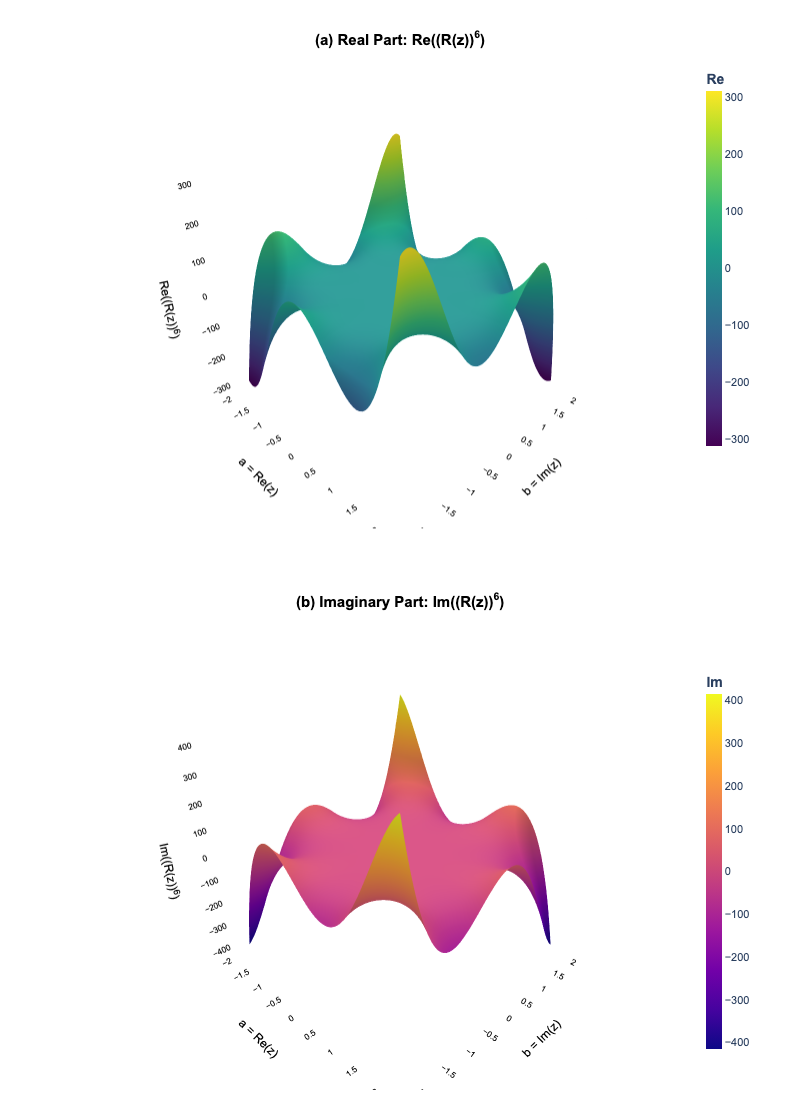}
        \caption*{$(R(z))^6$}
    \end{minipage}%
    \hspace*{-0.15\textwidth}
    \vspace{0.2cm}
    \caption{Surface-plot comparison of $R(z^6)$ and $(R(z))^6$. The real and imaginary components agree visually over the plotted domain, illustrating the identity $R(z^6) = (R(z))^6$ established algebraically above.}
    \label{fig:surface_plots}
\end{figure}

\subsection{Main Symmetry Theorem}

Equipped with Lemmas~\ref{lem:norm_preservation}--\ref{lem:iterate_commutation}, we state and prove the primary symmetry theorem governing generalized Mandelbrot sets.

\begin{theorem}[Symmetry Invariance]\label{thm:main_symmetry}
Let $\mathcal{M}_j$ denote the generalized Mandelbrot set of degree $j$, and let $T \in \{R_m, r_k\}$ be any symmetry transformation satisfying Lemmas~\ref{lem:norm_preservation} and \ref{lem:power_commutation}. If $c \in \mathcal{M}_j$, then $T(c) \in \mathcal{M}_j$.
\end{theorem}

\begin{proof}
Suppose $c \in \mathcal{M}_j$. By definition, the sequence of iterates $z_{n,c}$ remains bounded; that is, there exists some $B > 0$ such that $|z_{n,c}| < B$ for all $n \in \mathbb{N}_0$.

By Lemma~\ref{lem:iterate_commutation}, $z_{n, T(c)} = T(z_{n,c})$ for all $n$. Applying Lemma~\ref{lem:norm_preservation}:
\begin{equation*}
    |z_{n, T(c)}| = |T(z_{n,c})| = |z_{n,c}| < B \quad \forall n \in \mathbb{N}_0.
\end{equation*}
Thus, the sequence of iterates $z_{n, T(c)}$ remains bounded by $B$. Consequently, $T(c) \in \mathcal{M}_j$, establishing that $\mathcal{M}_j$ is invariant under $T$.
\end{proof}

\section{Morphological Evolution and Asymptotic Circularity}\label{app:morphology}
\label{sec:unique}

A striking feature of the generalized family $f(z) = z^j + c$ is the structural evolution of the main body as the polynomial degree $j$ increases. For any finite degree $j \ge 3$, the set exhibits $j-1$ distinct boundary lobes or "petals." These feature components directly reflect the discrete reflectional and rotational symmetries established in Section~\ref{sec:symmetry}.

Crucially, as $j$ increases, the angular spacing between adjacent symmetry axes ($\frac{2\pi}{j-1}$) decreases monotonically. This proliferation of lobes induces a visual and geometric smoothing effect: the discrete $(j-1)$-gonal boundary structures become progressively finer and densely packed around the complex unit circle $\mathbb{T} = \{c \in \mathbb{C} : |c| = 1\}$. 

In the asymptotic limit $j \to \infty$, the discrete rotational invariance transitions into continuous $S^1$ rotational symmetry, causing the macroscopic boundary of $\mathcal{M}_j$ to converge toward the unit disk $D_1(0)$. This structural compression provides the geometric intuition underlying the area convergence $A(\mathcal{M}_j) \to \pi$ demonstrated in Section~\ref{sec:area}.

Figure~\ref{fig:petal_structures} illustrates these structural features across degrees $j \in \{4, 6, 7, 10\}$.

\begin{figure}[htbp]
    \centering
    \makebox[\textwidth][c]{%
        \begin{minipage}{0.40\textwidth}
            \centering
            \includegraphics[width=0.75\textwidth]{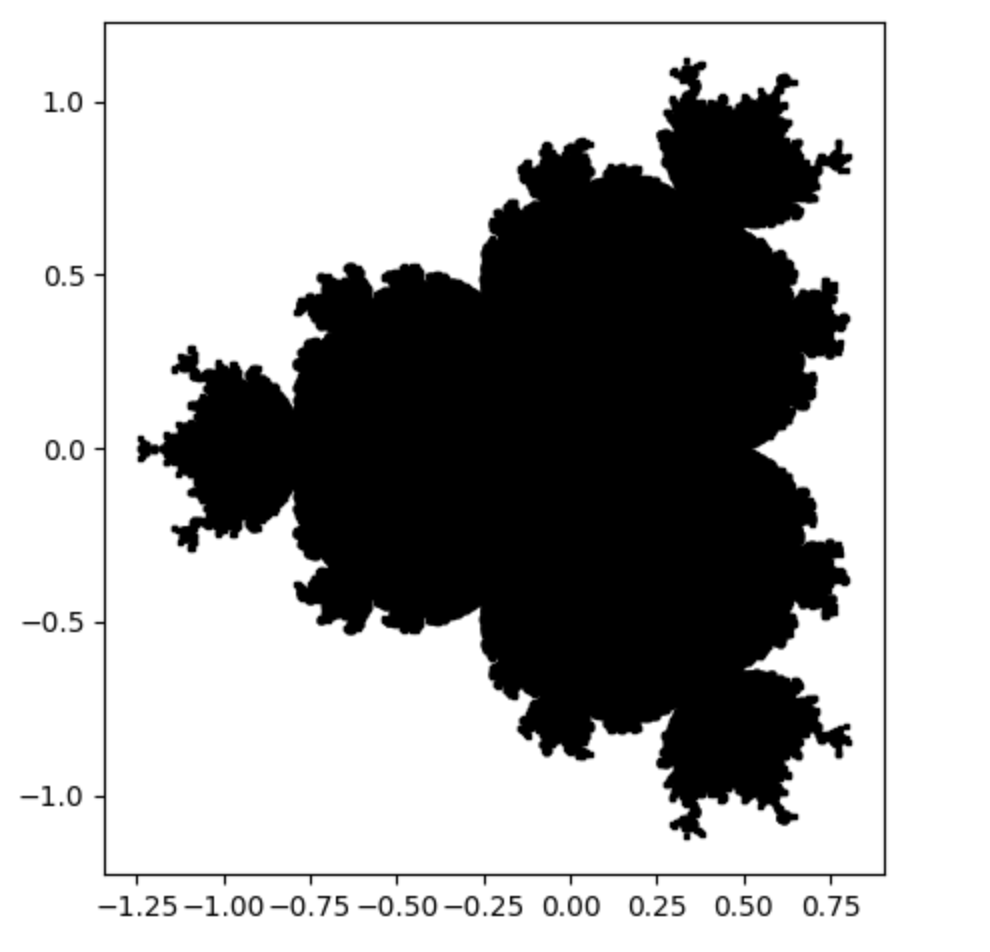}\\[0.2em]
            {\small $f(z) = z^4 + c$ \quad (3 lobes)}
        \end{minipage}\hspace{0.02\textwidth}%
        \begin{minipage}{0.40\textwidth}
            \centering
            \includegraphics[width=0.75\textwidth]{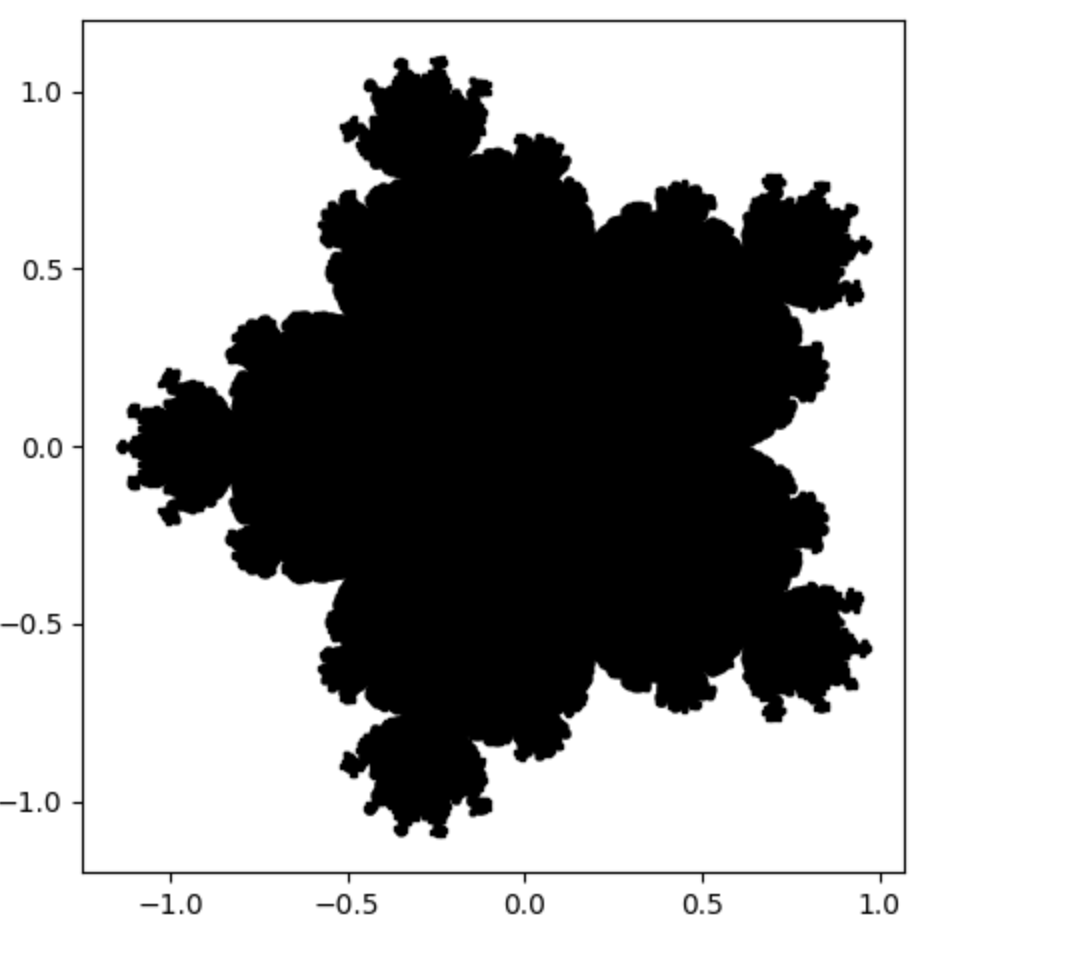}\\[0.2em]
            {\small $f(z) = z^6 + c$ \quad (5 lobes)}
        \end{minipage}%
    }
    
    \vspace{0.4em}
    
    \makebox[\textwidth][c]{%
        \begin{minipage}{0.40\textwidth}
            \centering
            \includegraphics[width=0.75\textwidth]{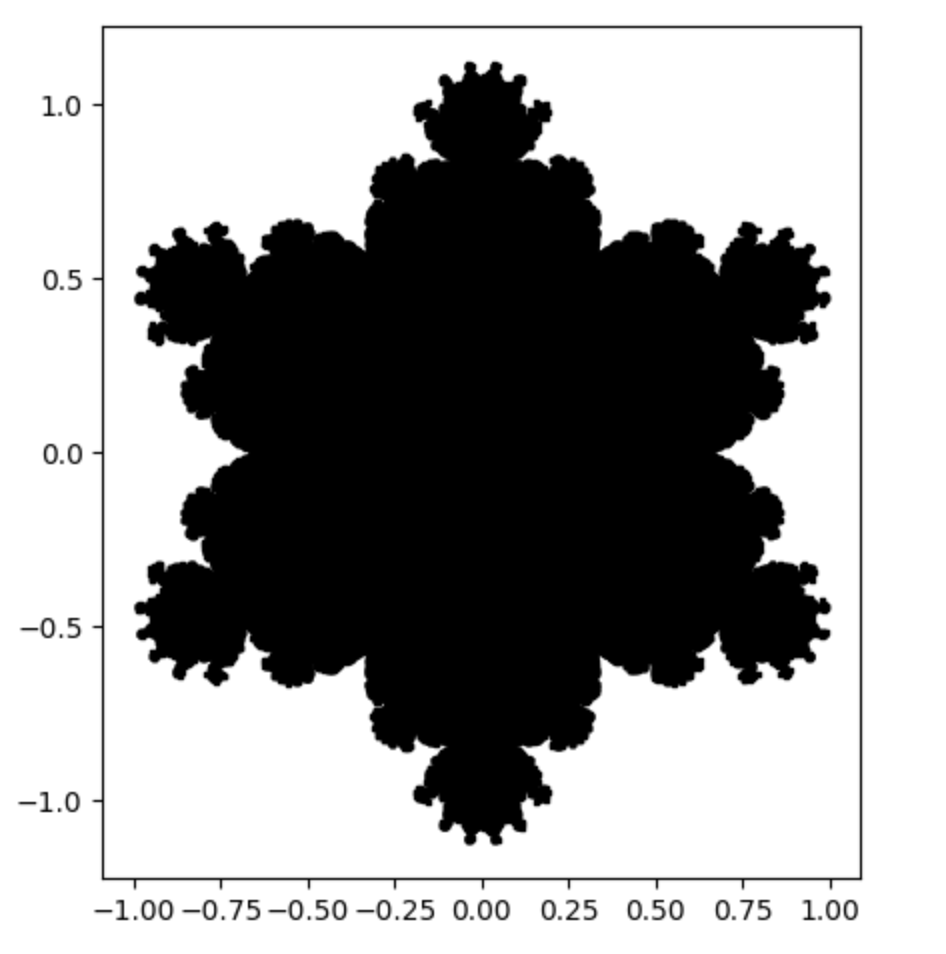}\\[0.2em]
            {\small $f(z) = z^7 + c$ \quad (6 lobes)}
        \end{minipage}\hspace{0.02\textwidth}%
        \begin{minipage}{0.40\textwidth}
            \centering
            \includegraphics[width=0.75\textwidth]{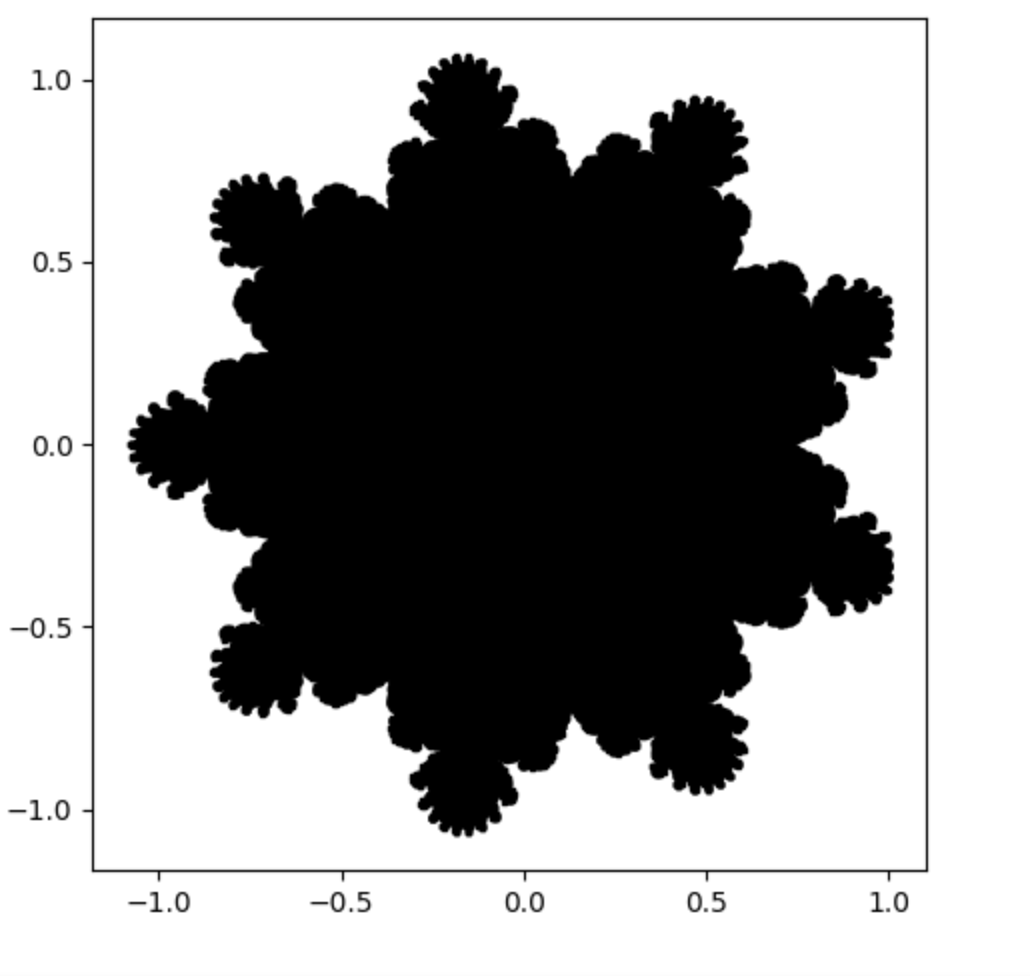}\\[0.2em]
            {\small $f(z) = z^{10} + c$ \quad (9 lobes)}
        \end{minipage}%
    }
    \caption{Morphological petal structures across increasing polynomial degrees $j$. Each set of degree $j$ exhibits $(j-1)$-fold rotational symmetry. As $j$ increases, the primary lobes multiply, sharpen, and uniformly compress toward the complex unit circle $\mathbb{T}$, driving the set's macroscopic geometry toward circularity.}
    \label{fig:petal_structures}
\end{figure}

\section{Additional Numerical Visualizations}\label{app:visuals}

\subsection{High-Degree Shape Evolution}
\textit{As $j$ increases, the Multibrot sets become increasingly well approximated by the unit disk.}\\

\noindent Representative Multibrot sets for several values of $j$ are shown in Figure~\ref{fig:high_degree_multibrots}, illustrating the increasingly circular geometry and smoother boundary as the degree increases. For smaller values of $j$, the boundary exhibits a pronounced gear-like structure. As $j$ increases, the curvature becomes progressively more uniform, visually approaching that of the unit circle.\par

\vspace{1em}
\begin{center}
    \begin{minipage}{0.45\textwidth}
        \centering
        \includegraphics[width=0.75\textwidth,height=0.10\textheight,keepaspectratio]{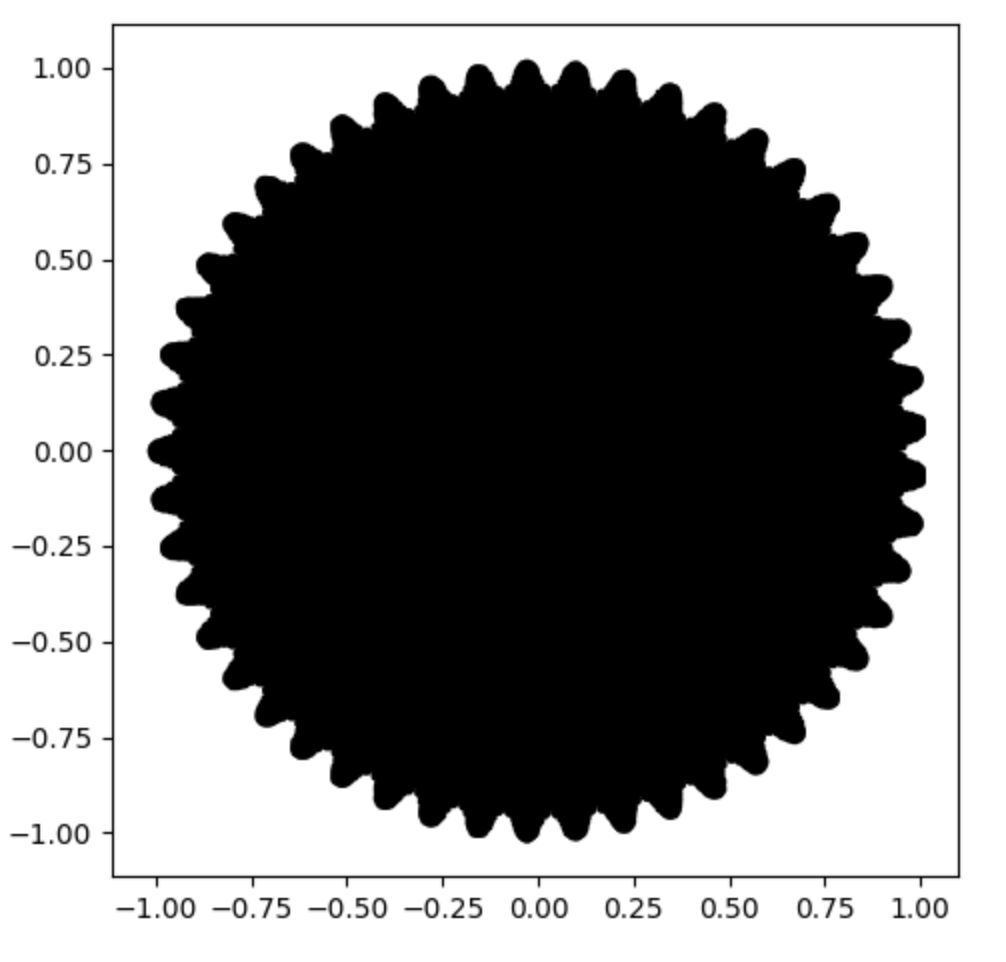}\\
        \vspace{0.2em}
        {\small $f(z)=z^{50}+c$}
    \end{minipage}%
    \hspace{0.05\textwidth}%
    \begin{minipage}{0.45\textwidth}
        \centering
        \includegraphics[width=0.75\textwidth,height=0.10\textheight,keepaspectratio]{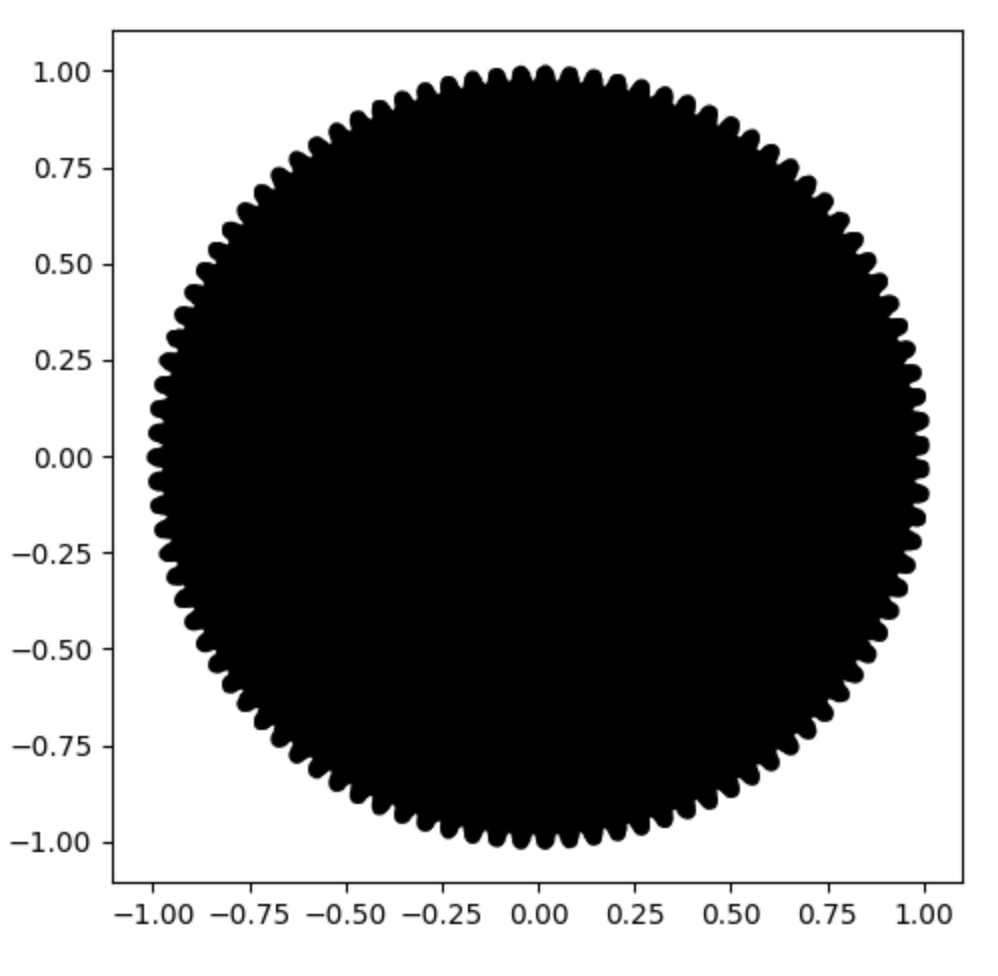}\\
        \vspace{0.2em}
        {\small $f(z)=z^{100}+c$}
    \end{minipage}

    \vspace{0.6em}

    \begin{minipage}{0.45\textwidth}
        \centering
        \includegraphics[width=0.75\textwidth,height=0.10\textheight,keepaspectratio]{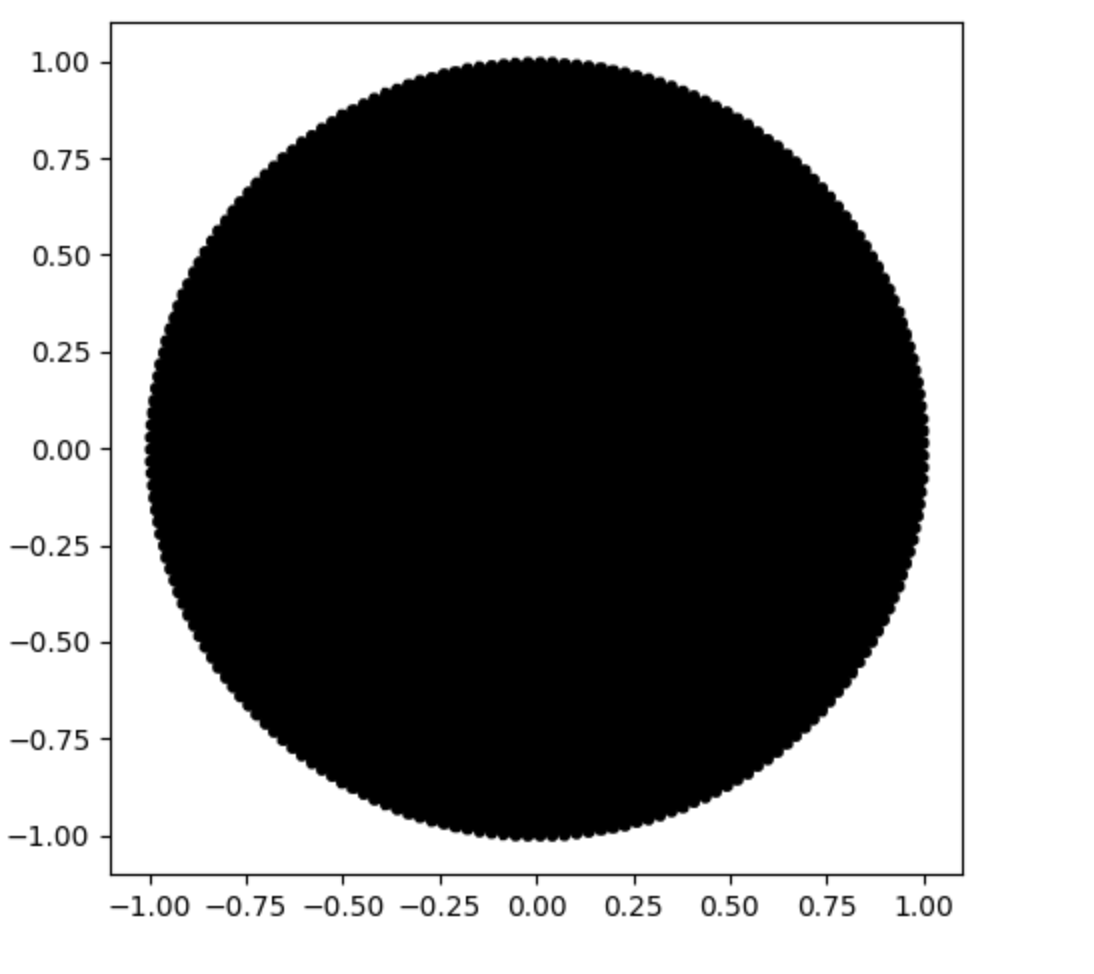}\\
        \vspace{0.2em}
        {\small $f(z)=z^{200}+c$}
    \end{minipage}%
    \hspace{0.05\textwidth}%
    \begin{minipage}{0.45\textwidth}
        \centering
        \includegraphics[width=0.75\textwidth,height=0.10\textheight,keepaspectratio]{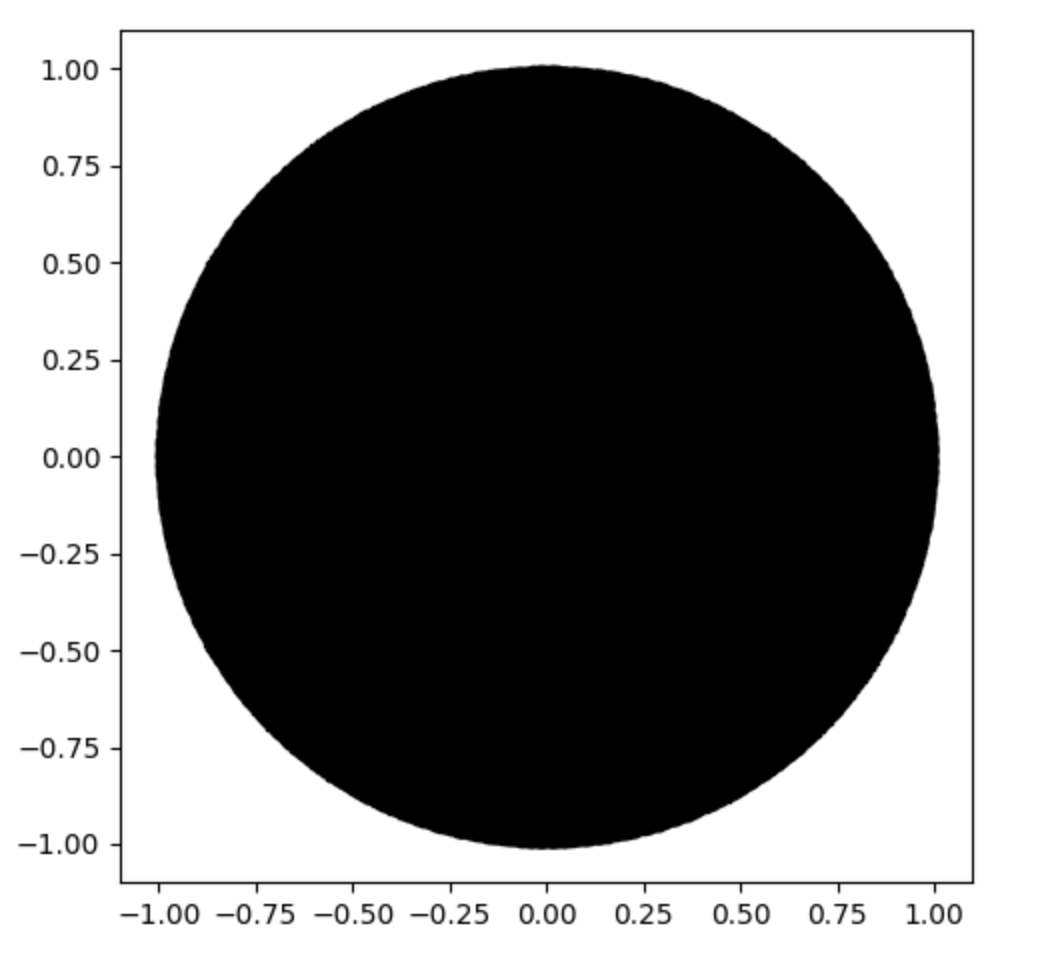}\\
        \vspace{0.2em}
        {\small $f(z)=z^{1000}+c$}
    \end{minipage}

    \vspace{0.8em}
    \captionof{figure}{Representative Multibrot sets for $f(z)=z^j+c$. As the degree increases, the sets become increasingly circular and their boundaries progressively smoother.}
    \label{fig:high_degree_multibrots}
\end{center}
\vspace{1em}

\subsection{Effective Radius: Numerical Observation}
\textit{As $j$ increases, the effective radius of the Multibrot set decreases.}\newline

The escape criterion established earlier shows that every Multibrot set is contained in a disk of radius $2$. However, the numerical experiments in Figure~\ref{fig:high_degree_multibrots} indicate that, for sufficiently large values of $j$, the sets become increasingly well approximated by the unit disk. Consequently, the effective radius appears to decrease toward $1$ as $j$ increases.

\section{Reference Algorithm}\label{app:implementation}

Algorithm~\ref{alg:multibrot_area} summarizes the escape-time pixel-counting
procedure used throughout the numerical experiments. The implementation used
for the reported computations evaluates grid rows in parallel using Numba,
but the underlying numerical procedure is independent of this implementation
choice.

\begin{algorithm}[htbp]
\caption{Escape-Time Pixel Counting for Multibrot Area Estimation}
\label{alg:multibrot_area}
\begin{algorithmic}[1]
\Require Degree $j\ge 2$, iteration limit $K$, sampling density $\rho$,
         rectangular domain $[x_{\min},x_{\max})\times[y_{\min},y_{\max})$
\Ensure Area estimate $\widehat{A}_{j,K,\rho}$

\State $W \gets \lfloor (x_{\max}-x_{\min})\rho \rfloor$
\State $H \gets \lfloor (y_{\max}-y_{\min})\rho \rfloor$
\State $\Delta x \gets (x_{\max}-x_{\min})/W$
\State $\Delta y \gets (y_{\max}-y_{\min})/H$
\State $N \gets 0$

\For{$n=0,\ldots,H-1$}
    \State $y \gets y_{\min}+n\Delta y$

    \For{$m=0,\ldots,W-1$}
        \State $x \gets x_{\min}+m\Delta x$
        \State $c \gets x+iy$
        \State $z \gets 0$
        \State $\textsc{escaped} \gets \textbf{false}$

        \For{$k=1,\ldots,K$}
            \State $z \gets z^j+c$

            \If{$|z|^2>4$}
                \State $\textsc{escaped} \gets \textbf{true}$
                \State \textbf{break}
            \EndIf
        \EndFor

        \If{\textbf{not} $\textsc{escaped}$}
            \State $N \gets N+1$
        \EndIf
    \EndFor
\EndFor

\State \Return $\widehat{A}_{j,K,\rho}
       \gets N\,\Delta x\,\Delta y$
\end{algorithmic}
\end{algorithm}
\captionof{figure}{Numba-accelerated Python implementation of the escape-time algorithm for generating generalized Mandelbrot sets. The script utilizes just-in-time compilation and parallelized execution to evaluate orbit divergence and compute membership within a defined bailout radius.}
\label{fig:escape}

\section{Additional Conformal Checks}\label{app:conformal_checks}
\begin{center}
\captionof{table}{First Laurent coefficients generated for the degree-$4$ Multibrot exterior map. The vanishing coefficients exhibit the expected sparsity pattern, while the leading nonzero coefficients provide direct checks of the generalized implementation.}
\label{tab:laurent_coefficients_j4}
\vspace{0.5em}
\begin{tabular}{c c @{\qquad} c c}
\toprule
Coefficient & Value & Coefficient & Value \\
\midrule
$b_{4,0}$  & $0$               & $b_{4,9}$  & $0$ \\
$b_{4,1}$  & $0$               & $b_{4,10}$ & $0$ \\
$b_{4,2}$  & $-\frac14$        & $b_{4,11}$ & $\frac{7}{2048}$ \\
$b_{4,3}$  & $0$               & $b_{4,12}$ & $0$ \\
$b_{4,4}$  & $0$               & $b_{4,13}$ & $0$ \\
$b_{4,5}$  & $-\frac1{32}$     & $b_{4,14}$ & $-\frac{31}{512}$ \\
$b_{4,6}$  & $0$               & $b_{4,15}$ & $0$ \\
$b_{4,7}$  & $0$               & $b_{4,16}$ & $0$ \\
$b_{4,8}$  & $0$               & $b_{4,17}$ & $-\frac{985}{65536}$ \\
\bottomrule
\end{tabular}
\end{center}

As a secondary qualitative check, we reconstructed the boundary contours of $\mathcal{M}_j$ directly from truncated Laurent series maps. Figure~\ref{fig:laurent_boundaries} shows the reconstructed exterior mappings for $j \in \{3, 4, 6\}$, faithfully reproducing the expected $(j-1)$-fold rotational symmetry.

\begin{figure}[htbp]
    \centering
    \begin{minipage}{0.32\textwidth}
        \centering
        \includegraphics[width=\textwidth]{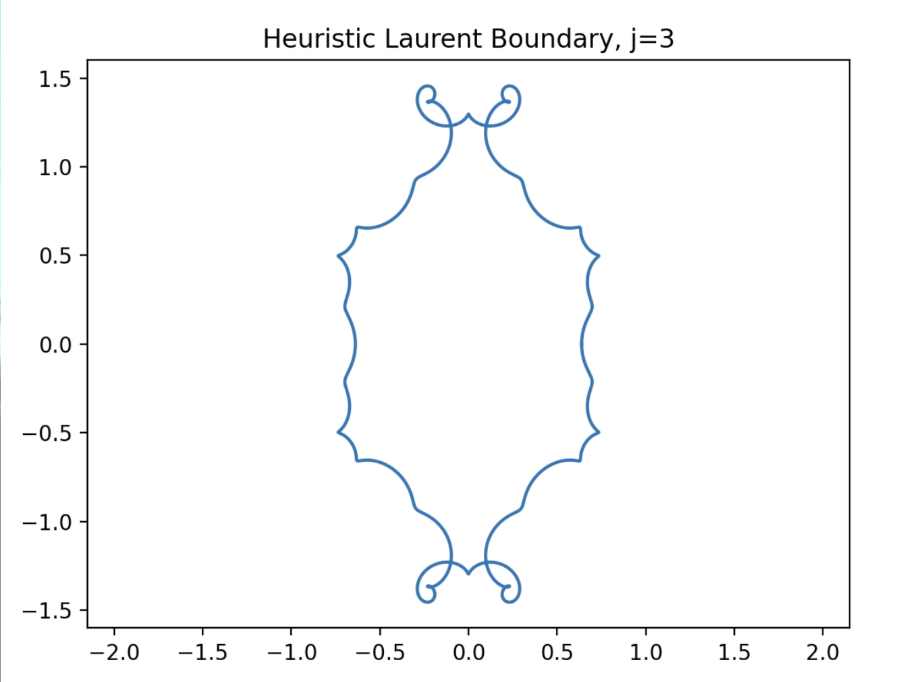}
        \caption*{$f(z) = z^3 + c$ \\ (2-gon boundary)}
    \end{minipage}\hfill
    \begin{minipage}{0.32\textwidth}
        \centering
        \includegraphics[width=\textwidth]{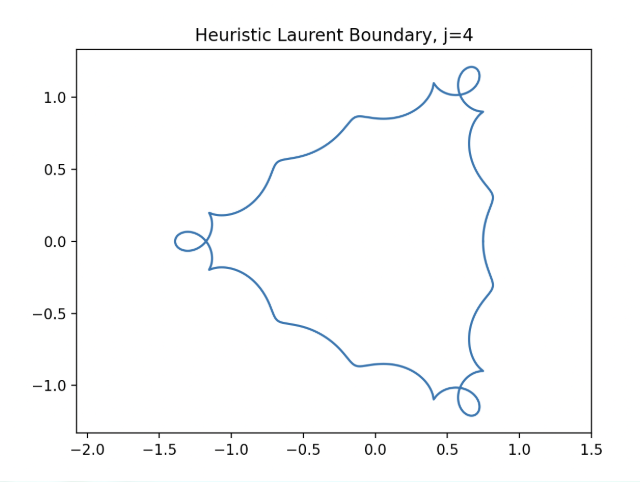}
        \caption*{$f(z) = z^4 + c$ \\ (3-gon boundary)}
    \end{minipage}\hfill
    \begin{minipage}{0.32\textwidth}
        \centering
        \includegraphics[width=\textwidth]{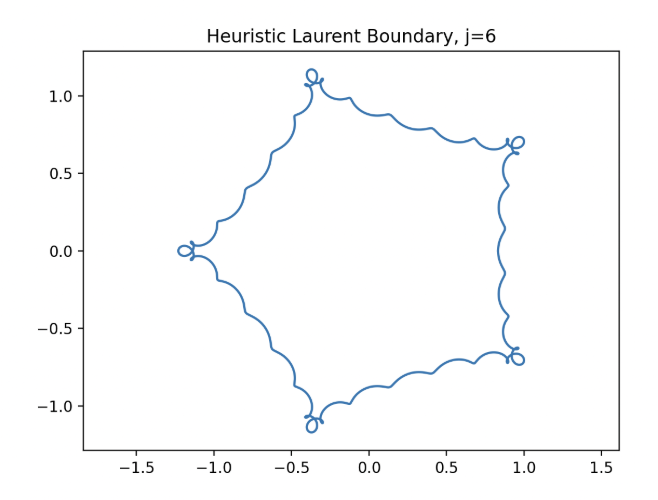}
        \caption*{$f(z) = z^6 + c$ \\ (5-gon boundary)}
    \end{minipage}
    \caption{Boundary approximations generated from truncated Laurent series. Reconstructions for degrees $j=3,4,$ and $6$ confirm that the generated coefficients accurately capture the geometric $(j-1)$-fold symmetry of the underlying Multibrot sets.}
    \label{fig:laurent_boundaries}
\end{figure}

\section{Supplementary Repository}
Additional development material, superseded experiments, and complete
symbolic derivations are archived at
\url{https://github.com/frostghost555/rohan_senapati-mandelbrot-appendix.git}.
The repository is supplementary; all claims needed for the main paper are
stated in the manuscript itself.

\bibliographystyle{plain}
\bibliography{references}

\end{document}